\documentclass[a4paper,12pt]{article}

\usepackage[latin9]{inputenc}
\usepackage{amsfonts}
\usepackage{amsmath}
\usepackage{amssymb}
\usepackage{amsthm}
\usepackage[english]{babel}
\usepackage{caption}
\usepackage{fontenc}
\usepackage{graphicx}
\usepackage{fancyhdr}
\usepackage{etoolbox}
\usepackage{url}
\usepackage{verbatim}
\usepackage{epstopdf}
\usepackage{color}

\makeatletter
\patchcmd{\thebibliography}{\chapter*{\bibname}\@mkboth{\MakeUppercase\bibname}{\MakeUppercase\bibname}}{\section*{References}}{}{}
\makeatother

\theoremstyle{plain}
\newtheorem{thm}{Theorem}
\newtheorem{prop}[thm]{Proposition}
\newtheorem{lem}[thm]{Lemma}

\newtheorem{defn}{Definition}
\newtheorem{ass}{Assumption}
\newtheorem{asss}{Assumptions}
\newtheorem{oss}{Remark}

\newcommand{\FF}{\phantom{\frac{|}{|}}}
\newcommand{\ff}{\phantom{\tfrac{|}{|}}}

\title{A Feynman-Kac type formula for a fixed delay CIR model}
\author{Federico Flore\thanks{Dipartimento di Economia Aziendale - Universit\`{a} di Roma Tre, Via Silvio D'Amico,77, 00145 Roma, Italy.}
\and Giovanna
Nappo\thanks{Dipartimento di Matematica - Universit\`{a} di Roma ``Sapienza", Piazzale A. Moro 5, 00185 Roma, Italy.}}
\date{}
\begin{document}
\maketitle
\begin{abstract}
\noindent Stochastic delay differential equations (SDDE's) have been used for financial modeling. In this article, we study a SDDE obtained by the equation of a CIR process, with an additional fixed delay term in drift; in particular, we prove that there exists a unique strong solution (positive and integrable) which we call fixed delay CIR process. Moreover, for the fixed delay CIR process, we derive a Feynman-Kac type formula, leading to a generalized exponential-affine formula,  which is used to determine a bond pricing formula when the interest rate follows the delay's equation. It turns out that, for each maturity time $T$, the instantaneous forward rate is an affine function (with time dependent coefficients) of the rate process and of an auxiliary process (also depending on $T$). The coefficients satisfy a system of deterministic differential equations.
\end{abstract}
\section{Introduction}\label{sec:introduction}
In a seminal paper (\cite{CIR:SIR}) Cox, Ingersoll and Ross proposed a model for the interest rates, that has  found considerable use also as a model for volatility and other financial quantities. The model is named  Cox-Ingersoll-Ross (CIR) model or mean-reverting square root process (and is also known as Bessel-square process or Feller process),
and is expressed as the solution of the following stochastic differential equation
\begin{equation}\label{eq:CIR}\begin{cases}
dr(t)=a_r(\gamma_r-r(t))dt+\sigma_r\sqrt{r(t)}dW_r(t),\\
r(t_0)=r_0,
\end{cases}\end{equation}
where $W_r(t)$ is a standard Brownian motion and $a_r$, $\gamma_r$ and $\sigma_r$ are positive constants. There are three appealing properties why this model is used so widely.
First, Eq.~\eqref{eq:CIR} has a unique nonnegative solution for any positive initial value with probability one, which is very important since this equation is often used to model the interest rate (or volatility). Second, it is mean-reverting and the expectation of $r(t)$ converges to $\gamma_r$, the so-called long-term value, with the speed $a_r$ (see, e.g., Higham and Mao \cite{HighamMao} and the references therein).
Third, since its incremental variance is proportional to the current value, one can compute explicitly the term structure.\\
In order to better capture the properties of the empirical data, there are many extensions of the CIR model, e.g., Chan, Karolyi, Longstaff and Sander \cite{JOFI:JOFI4011} generalize the CIR model as
\[dr(t)=a_r(\gamma_r-r(t))dt+\sigma_rr^\theta(t)dW_r(t),\]
where $\theta \geq \frac{1}{2}$. As explained in  Hull and White \cite{Hull01101990}, another generalization of the CIR model can be obtained so that it is consistent with both the current term structure of interest rates and either the current volatilities of all spot interest rates or the current volatilities of all forward interest rates. In their paper, the authors consider the following version of CIR model with time-dependent parameters; i.e.,~they consider the following model
\begin{equation}\label{eq:CIR_HullWhite}\begin{cases}
dr(t)=\left[\varphi(t)+a_r(t)\left(\gamma_r-r(t)\right)\right]dt+\sigma_r(t)r^\theta(t)dW_r(t),\\
r(t_0)=r_0.
\end{cases}\end{equation}

From the financial point view, the interest has focused in models where the underlying asset's dynamics is given by a stochastic delay differential equation (SDDE). In this regard, we can cite  Arriojas, Hu, Mohammed and Pa \cite{ArrHuMohaPap}: the authors consider a market where the evolution of the stock price $S(t)$ is described by the following equation
\begin{equation}\label{eq:AHMP_equation}\begin{cases}
dS(t)&=f(t,S_t)dt+ g(S(t-b))S(t)dW(t),\quad t \in [0,T]\\
S(t) &= \varphi(t), \qquad t \in [-\tau, 0]
\end{cases}\end{equation}
where the drift coefficient $f:[0,T]\times \mathcal{C}([-\tau,0];\mathbb{R}) \rightarrow \mathbb{R}$ is a given continuous functional, $S_t\in\mathcal{C}([-\tau,0];\mathbb{R})$ stands for the segment process $S_t(u):=S(t+u)$, $u\in[-\tau,0]$, the diffusion coefficient $g$ is a continuous function, the parameters $\tau$ and $b$ are positive constants,  with $\tau \geq b$, and the process $\varphi(t)$ is $\mathcal{F}_0$-measurable with respect to the Borel $\sigma$-algebra of $\mathcal{C}([-\tau,0];\mathbb{R})$. Under the suitable hypotheses on $f$ and $g$, the authors prove that Eq.~\eqref{eq:AHMP_equation} has a pathwise unique solution. Furthermore when the drift coefficient $f(t,\eta)$ is equal to $\mu\eta(-a)\eta(0)$, i.e., $f(t,S_t)=\mu S(t-a)S(t)$, where $a$ is positive constant; setting $\tau=\max\{a,b\}$, the authors develop an explicit formula for pricing European options.
 Moreover, the authors give an alternative model for the stock price dynamics with variable delay and, also in this case, are able to develop a formula for the option price.\\
Although is not present a financial application, we can also cite  Wu, Mao and Chen \cite{WuMaoChen} where the authors generalize the Euler-Maruyama (EM) scheme to the following model with delay term in diffusion coefficient
\begin{equation}\label{eq:WMC}\begin{cases}
dS(t) &= \lambda(\mu-S(t))dt+ \sigma S(t - \tau )^\gamma\sqrt{S(t)}dW(t),\quad t\in[0,T]\\
S(\theta) &= \xi(\theta), \qquad \theta \in [-\tau , 0],
\end{cases}\end{equation}
where $\xi \in \mathcal{C}([- \tau , 0];\mathbb{R}^+)$,
and prove the strong convergence of the EM approximate solutions to the (unique and nonnegative) solution of~\eqref{eq:WMC}.\\

In this paper, we assume that the spot rate satisfies the following SDDE with fixed delay in the drift coefficient
\begin{equation}\label{eq:CIRdelay}\begin{cases}
dr(t)=[a_r(\gamma_r(t)-r(t))+b_rr(t-\tau)]dt+\sigma_r\sqrt{r(t)}dW_r(t),\\
r(t)=r_0(t) \qquad t_0-\tau\leq t \leq t_0,\end{cases}\end{equation}
with $a_r$, $\sigma_r$ positive constants,  $\gamma_r(t)$ a positive function, bounded on bounded time intervals, and $b_r\geq 0$, so that our model~\eqref{eq:CIRdelay} is a generalization of the classical CIR model~\eqref{eq:CIR}.
\\

In the first part of this paper, we focus our interests on  existence and uniqueness of the solution of Eq.~\eqref{eq:CIRdelay}. (We will refer to the unique solution as the fixed delay CIR process.) In the remaining part of the paper, we derive a  Feynman-Kac type formula in order to determine a formula for the unitary zero-coupon bond (uZCB) price  and a formula for the instantaneous forward rate.\\

The paper is organized as follows. Section \ref{sec:Existence} is devoted to state the properties of our fixed delay CIR process; in particular, under the Assumptions~\ref{asss:hp_r}, Eq.~\eqref{eq:CIRdelay} has a unique and nonnegative solution. Moreover, if the Feller condition \eqref{eq:feller's_condition_delay} holds, the solution is positive (see Theorem~\ref{thm:ExistenceUniqueness}). Furthermore, under Assumption~\ref{ass:integrabilityUnidimensional}, i.e., if the initial segment $r_0(t) $ satisfies an  integrability condition, uniformly on the interval $[t_0-\tau,t_0]$ (see \eqref{eq:SCintegrability}), then the solution is integrable, together with its supremum on any bounded interval (see Proposition~\ref{prop:integrabilityCIR}  and Remark~\ref{oss:SCintegrability}). The aim of Section~\ref{sec:TSBV} is to determine a pricing formula for uZCB. Under Assumption~\ref{ass:MarketPriceOfRisk-OneDim},  there exists a risk-neutral probability measure and hence, the financial market is arbitrage-free (see   Theorem~\ref{thm:existenceQ} and Remark~\ref{oss:MeasureQ}). Another important result of this section is an extension of the well-known Feynman-Kac formula (see Theorem~\ref{thm:term_structure_one-dimensional_CIR_delay}) that is used to determine a pricing formula for uZCB. In Section~\ref{sec:FR}, we recall the definition of instantaneous forward rate and prove that if the  spot rate is a fixed delay CIR process, then the instantaneous forward rate is a linear function of the spot rate and of another suitable process (see Theorem~\ref{thm:term_structure_derivate_delay-one-dim}); this result extends the usual formula of the CIR instantaneous forward rate. Appendix~\ref{app:ProofsCIR} is devoted to the proofs of some technical results.

\section{Properties of the Fixed Delay CIR Process}\label{sec:Existence}
Throughout this paper, unless otherwise specified, we use the following notations. Let $(\Omega,\mathcal{F}, \mathbb{P})$ be a complete probability space with a right continuous filtration $\{\mathcal{F}_t\}_{t\geq t_0}$ and let $\mathcal{F}_{t_0}$ contain all $\mathbb{P}$-null sets. Let $W_r(t)$ be a scalar Brownian motion defined on this probability space.
\\
Our aim is to use Eq.~\eqref{eq:CIRdelay}  as a model for interest rate, volatility and other financial quantities, therefore, besides existence and uniqueness of the solution, it is crucial that the solution be positive.

Actually, we examine the equation
\begin{equation}\label{eq:CIRdelay2}\begin{cases}
dr(t)=[a_r(\gamma_r(t)-r(t))+b_r r(t-\tau)]dt+\sigma_r\sqrt{|r(t)|}dW_r(t)\\
r(t)=r_0(t) \qquad t_0-\tau\leq t \leq t_0,\end{cases}
\end{equation}

Throughout this paper we make the following assumptions.
\begin{asss}\label{asss:hp_r}${}$
\begin{description}
\item[(i)]The process $W_r(t)$, $t\geq t_0$, is a Brownian motion with respect to the filtration $\mathcal{F}_t$,  so that
    $\mathcal{F}_{t_0}$ is independent of  natural filtration $\mathcal{F}^{W_r}_t$;
    \item[(ii)]the parameters $a_r$ and $\sigma_r$ are positive constants, and $b_r$ is a nonnegative constant;\\
\item[(iii)]the segment $r_0(\cdot)$ is a positive continuous random function on $[t_0-\tau, t_0]$
    such that
    \begin{equation}
    \label{eq:MeasurableCond-CIR}
    \int_{t_0-\tau}^{t}r_0(u) du < +\infty,\;\;\mathbb{P}\text{-a.s.};\end{equation}
    moreover, we require that $r_0(t)$ is a $\mathcal{F}_{t_0}$-measurable for $t_0-\tau\leq t\leq t_0$ and therefore $\sigma\{\left(r_0(u)\,; u\in[t_0-\tau,t_0]\right)\}$ is independent of $\mathcal{F}^{W_r}_t$, $t\geq t_0$;
\item[(iv)] the deterministic function $\gamma_r(t)$ is measurable, positive, and bounded on every bounded interval.
\end{description}
\end{asss}

\begin{thm}\label{thm:ExistenceUniqueness}${}$\\
Under the Assumptions~\ref{asss:hp_r}, the Eq.~\eqref{eq:CIRdelay2} admits a unique solution and the solution is nonnegative.
\\
Assume that $0\leq \underline{b}_r \leq b_r$, and that the initial segment $\underline{r}_0(t)$ is  such that $\underline{r}_0(t)\leq r_0(t)$, for $t_0-\tau\leq t\leq t_0$.  Let $r^{(b_r)}(t)$ be the solution of~\eqref{eq:CIRdelay2}, and  $r^{(\underline{b}_r)}(t)$ be the solution of~\eqref{eq:CIRdelay2}, with $\underline{b}_r$ and the initial segment $\underline{r}_0(t)$ in place of $b_r$  and  $r_0(t)$, respectively. Then
\begin{equation}\label{comparison}
r^{(\underline{b}_r)}(t) \leq r^{(b_r)}(t), \quad \text{for all $t\geq t_0$}.
\end{equation}
If moreover the following inequality holds
\begin{equation}\label{eq:feller's_condition_delay}
\sigma_r^2\leq2a_r\gamma_r(t) \quad\text{for all $t\geq t_0$},
\end{equation}
then the process $r^{(b_r)}(t)$ is positive.
\end{thm}
\begin{proof}${}$\\
Existence, uniqueness and the comparison results can be achieved by induction on the intervals $[t_0+k \tau, t_0+(k+1)\tau]$, $k \geq 0$. In the first interval $[t_0, t_0+\tau]$ Eq.~\eqref{eq:CIRdelay} is a particular case of the equations
 \begin{equation}\label{eq:DDequation}dX(t)=(2\beta X(t)+\delta(t))dt+g(X(t))dW(t)\quad\text{for all $t\in[0,+\infty)$},
 \end{equation}
 studied in   Deelstra and Delbaen~\cite{deelstra_delbaen}, with $2\beta=-a_r$, $\delta(t)=b_r\,r_0(t-\tau)+ a_r \gamma_r(t)$  and $g(x)=\sigma_r \sqrt{|x|}$. As observed in the latter paper, `` \emph{Eq.~\eqref{eq:CIRdelay2} is a Dol\'{e}ans, Dade and Protter's equation, and it is shown by Jacod~\cite{Jacod}  that there
exists a unique strong solution. Extending comparison results as in Karatzas and Shreve~\cite{karatzas_shreve}  (p.\@ 293)
or Revuz-Yor~\cite{RevuzYor}  (p.\@ 394),
it is easy to check that  the solution remains nonnegative a.s.\@ (see, e.g.,
Deelstra~\cite{DeelstraPhD}),}'' and that inequality~\eqref{comparison} holds. Finally,  taking $\underline{b}_r=0$, (so that the initial segment $\underline{r}_0(t)$ is irrelevant)  under the Feller  condition~\eqref{eq:feller's_condition_delay}  the process $r^{(\underline{b}_r)}(t)$ is the classical CIR model and is   positive a.s. Then the comparison~\eqref{comparison} immediately implies that  also $r^{(b_r)}(t)$ remains positive a.s.
\end{proof}

\begin{oss}\label{oss:UiL}${}$\\
The previous result guarantees existence and strong uniqueness  of  Eq.~\eqref{eq:CIRdelay}, for all initial segment with continuous paths. Then Yamada-Watanabe theorem (see, e.g., Cherny~\cite{Cherny}),  implies weak uniqueness, i.e., uniqueness in distribution of the solutions.
\end{oss}
 The next result deals with the integrability of a fixed delay CIR process, i.e., the unique  solution of Eq.~\eqref{eq:CIRdelay}.
\begin{prop}\label{prop:integrabilityCIR}${}$\\
Suppose that on a probability space $(\Omega,(\mathcal{F}_t)_{t\geq t_0},\mathbb{P})$, the process $r(t)$ is a fixed delay CIR process, defined by Eq.~\eqref{eq:CIRdelay}.
If
\begin{equation}\label{eq:IntCondr1}
\int_{t_0-\tau}^{t_0}\mathbb{E}\left[r_0(u)\right]du<+\infty,
\end{equation}
and
\begin{equation}\label{eq:IntCondr2}
\mathbb{E}\left[r_0(t_0)\right]<+\infty,\end{equation}
then,
\begin{enumerate}
\item for all $t\geq t_0$
\begin{equation}\label{eq:integr-uniform}
\mathbb{E}\left[\sup_{t_0\leq u \leq t}r(u)\right] < \infty,
\end{equation}
\item the following formula holds, for all $t\geq t_1\geq t_0$
       \begin{equation}\label{eq:ExpValDelay-CIR}
       \mathbb{E}\left[r(t)\right]=\mathit{e}^{-a_r(t-t_1)}\mathbb{E}\left[r(t_1)\right]
       +\int_{t_1}^t\mathit{e}^{-a_r(t-u)}\left(a_r\gamma_r(u)+b_r\mathbb{E}\left[r(u-\tau)\right]
       \right)\,du.\end{equation}
\end{enumerate}
\end{prop}
\begin{proof}${}$\\
Similarly to the previous Theorem~\ref{thm:ExistenceUniqueness}, the proof can be achieved by induction on the intervals $[t_0+k \tau, t_0+(k+1)\tau]$ by proving that
\begin{equation}\label{eq:integr-uniform-k}
\mathbb{E}\left[\sup_{t_0+k\tau\leq u \leq t_0+ (k+1)\tau}r(u)\right] < \infty.
\end{equation}
  For $k=0$
   the thesis follows by Lemma~1 in Deelstra and Delbaen~\cite{DDLongTerm}: hypotheses~\eqref{eq:IntCondr1} and~\eqref{eq:IntCondr2}, imply  the integrability condition~\eqref{eq:integr-uniform-k} (with $k=0$).
It is important to stress that this implication does not appear in the statement of Lemma~1, but it is one of the steps in the  proof of the above mentioned result (see p.\@ 166 in~\cite{DDLongTerm}).
The induction step  from $k$ to $k+1$  follows by observing that  condition~\eqref{eq:integr-uniform-k} is stronger than necessary.
\end{proof}
\begin{oss}${}$\\\label{oss:SCintegrability}
The following integrability condition (uniform on the interval $[t_0-\tau,t_0]$)
\begin{equation}\label{eq:SCintegrability}
\sup_{t\in[t_0-\tau,t_0]}\mathbb{E}\left[r_0(t)\right]<+\infty,
\end{equation}
ensures that the assumptions \eqref{eq:IntCondr1} and \eqref{eq:IntCondr2} of Proposition \ref{prop:integrabilityCIR} hold true.
\end{oss}
\section{Term Structure for Bond Valuation}\label{sec:TSBV}
Term structures of interest rates describe the relation between interest rates and bonds with different maturity times. We recall that,
by convention, a unitary zero-coupon bond  with maturity $T<+\infty$, pays one unit of cash at the prescribed date $T$ in the future, and its price is denoted by $B(t,T)$, at time $t\leq T$; it is thus clear that, necessarily, $B(T,T)=1$ for any maturity $T$.\\
At time $t$, the yield to maturity $R(t,T)$ of the uZCB $B(t,T)$ is the continuously compounded (constant) rate of return that causes the bond price to rise to one a time $T$, i.e.,
$$
B(t,T)\mathit{e}^{(T-t)R(t,T)}=1,
$$
or, solving for the yield,
\begin{equation}\label{eq:YieldToMaturity}
R(t,T):=-\frac{1}{T-t}\ln(B(t,T)).
\end{equation}
For a fixed time $t$, the curve $T\,\mapsto\,R(t,T)$ determines the term structure of interest rates.\\
\begin{defn}\label{defn:ShortRate}
The (instantaneous) spot rate $r(t)$ is defined by
\begin{equation}\label{eq:ShortRate}
r(t):=\lim_{T\rightarrow t}R(t,T).
\end{equation}
\end{defn}
In an Arbitrage-free market, the Bond price is given by
\begin{equation}\label{eq:BondPrice}
B(t,T)=\mathbb{E}^{\mathbb{Q}}\left[\mathit{e}^{-\int_t^Tr(u)du}\bigg{|}
\mathcal{F}_t\right]\quad\text{for all $t\in[t_0,T]$},\end{equation}
where $\mathbb{E}^\mathbb{Q}$ is the expectation with respect to the risk-neutral measure used by market, $r(t)$ is the $\mathcal{F}_t$-adapted instantaneous interest rate
(see, e.g., Lamberton and Lepeyre~\cite{lamberton_lapeyre} or Musiela and Rutkowski~\cite{MusieleRutkowski}).

 When $r(t)$ is a classical CIR process, the existence of a risk-neutral probability measure $\mathbb{Q}$ is guaranteed by the uniqueness of the martingale problem (see, e.g., Theorem $2.4$ in Cheredito, Filipovi\'c and Yor~\cite{ChereditoFilYor2005}).
Cheredito, Filipovi\'c and Kimmel~\cite{cheridito2007market} prove the existence of a risk-neutral probability measure $\mathbb{Q}$ taking advantage of the uniqueness in law of the involved processes (see Theorem~$1$ in Cheredito, Filipovi\'c and Kimmel~\cite{cheridito2007market}). We extend the result of Cheredito, Filipovi\'c and Kimmel to prove the existence of $\mathbb{Q}$ for a fixed delay CIR process.
\\

In what follows we assume that the interest rate $r(t)$ is the fixed delay CIR process solution of Eq.~\eqref{eq:CIRdelay}, and since we  need to consider different Brownian motions, under different probability measure, we will use the notation $W^\mathbb{P}_r(t)$ instead of $W_r(t)$.
\\

\begin{thm}\label{thm:existenceQ}${}$\\
Under Assumptions~\ref{asss:hp_r},  let $r(t)$ be the solution of
$$\begin{cases}
r(t)=r_0(t_0)+ \int_{t_0}^t \mu^\mathbb{P}(s,r(\cdot))ds+ \int_{t_0}^t \sigma_r
  \sqrt{r(s)} \, dW^\mathbb{P}_r(s),\quad t\in [t_0,T],\FF
  \\
  r(t)=r_0(t), \quad t\in [t_0-\tau, t_0],\FF
 \end{cases}
$$
where
$$\mu^\mathbb{P}(t,x(\cdot))=a_r(\gamma_r(t)-x(t)))+b_r x(t-\tau)$$
 and the parameters   $a_r$ and $\sigma_r$ and the function~$\gamma_r(t)$ satisfy the Feller condition~\eqref{eq:feller's_condition_delay}.\\
Assume  that $b^\mathbb{Q}_r\geq 0$,   $a^\mathbb{Q}_r>0$,
 the function $\gamma^\mathbb{Q}_r(t)$  is measurable, positive,  and bounded on every bounded interval,
and finally that  the  Feller condition
 \begin{equation}\label{eq:feller's_condition_delay-Q}
\sigma_r^2\leq2a^\mathbb{Q}_r\gamma^\mathbb{Q}_r(t)
\end{equation}
holds.
 Consider the functional $\mu^\mathbb{Q}(t,x(\cdot))$ so defined
\begin{equation}\label{eq:mu-Q}
\mu^\mathbb{Q}(t,x(\cdot)):=a^\mathbb{Q}_r\big(\gamma^\mathbb{Q}_r(t)-x(t)\big)+b^\mathbb{Q}_r x(t-\tau).
\end{equation}

Then, there exists a probability measure $\mathbb{Q}$, such that
 \begin{enumerate}
\item  $\mathbb{Q}=\mathbb{P}$ on   $(\Omega, \mathcal{F}_{t_0})$,
 \item for each $T>t_0$,
$\mathbb{Q}$ is equivalent to $\mathbb{P}$ on $(\Omega, \mathcal{F}_T)$,
\item
there exists  a process $W^\mathbb{Q}_r$, which is a Brownian motion under $\mathbb{Q}$, and such that
$$
r(t)=r_0(t_0)+ \int_{t_0}^t \mu^\mathbb{Q}(s,r(\cdot))ds+ \int_{t_0}^t \sigma_r
  \sqrt{r(s)} \, dW^\mathbb{Q}_r(s),\quad t\in [t_0,T].
$$
\end{enumerate}
Finally the probability measure $\mathbb{Q}$ on $(\Omega, \mathcal{F}_T)$ is defined by $d\mathbb{Q}=Z_T\, d\mathbb{P}$, where
$$
  Z_t  :=\exp\left\{ -\int_{t_0}^t \xi_r(s,r(\cdot)) \, dW^\mathbb{P}(s)-  \frac{1}{2}\,\int_{t_0}^t \xi^2_r(s,r(\cdot))\, ds\right\},
$$
with
\begin{align}\notag
\xi_r(t,r(\cdot))&:= \frac{\mu^\mathbb{P}(t,r(\cdot))-\mu^\mathbb{Q}(t,r(\cdot))}{\sigma_r\sqrt{r(t)}}
\\&= \frac{a_r\gamma_r(t) - a^\mathbb{Q}_r\gamma_r^\mathbb{Q}(t) -(a_r- a_r^\mathbb{Q}) r(t) + (b_r-b_r^\mathbb{Q})r(t-\tau)}{\sigma_r\sqrt{r(t)}}.
\label{eq:xi-r-t-r}
\end{align}

\end{thm}
\begin{proof}${}$\\
(See Appendix~\ref{app:ProofsCIR}).
\end{proof}

Besides  Assumptions~\ref{asss:hp_r}, we now assume some further conditions.
\begin{ass}\label{ass:integrabilityUnidimensional}${}$\\
The fixed delay CIR process $r(t)$ satisfies the  integrability condition~\eqref{eq:SCintegrability}.
\end{ass}
\begin{ass}\label{ass:MarketPriceOfRisk-OneDim}${}$\\
The market price of risk is a one-dimensional process $\xi_r(t)$ adapted with respect to the filtration~$\mathcal{F}_t$ and right continuous such that $\xi_r(t)$ is given by

\begin{equation}\label{eq:MarketPriceOfRisk-OneDim}
\xi_r(t)=\frac{\sqrt{r(t)}}{\sigma_r}\,\psi^r,\end{equation}
where  $\psi^r$ is a nonnegative constant, i.e., the risk-neutral measure $\mathbb{Q}$ is defined as the measure on the same space $(\Omega,\mathcal{F}_T)$ with Radon-Nikodym derivate given by
\begin{equation}\label{eq:Q_measure-CIR}\frac{d\mathbb{Q}}{d\mathbb{P}}\bigg{|}_{\mathcal{F}_T}=Z^{\xi_r}(T),\quad\text{that is}\quad \mathbb{Q}(F)=\int_FZ^{\xi_r}(T)\mathbb{P}(d\omega),\; F\in\mathcal{F}_T,
\end{equation}
where $Z^{\xi_r}(t)$ is the following one-dimensional $\mathbb{P}$-martingale process
\begin{equation}\label{eq:martingala_esponenziale_unidimensionale}
Z^{\xi_r}(t)=\mathit{e}^{-\int_{t_0}^t\xi_r(s)dW_r(s)+\frac{1}{2}\int_{t_0}^t\xi^2_r(s)ds}\quad t\geq t_0.\end{equation}
\end{ass}

\begin{oss}\label{oss:MeasureQ}${}$\\
Under Assumption~\ref{ass:MarketPriceOfRisk-OneDim}, by construction, the probability measures $\mathbb{P}$ and $\mathbb{Q}$ are equal on $\mathcal{F}_{t_0}$.
Consequently, Assumption~\ref{ass:integrabilityUnidimensional}  holds true also w.r.t.\@ the risk-neutral measure~$\mathbb{Q}$, and the process~$r(t)$ is again a fixed delay CIR process  w.r.t.\@ the measure $\mathbb{Q}$, with dynamics described by
\begin{equation}\label{eq:IRQequation}\begin{cases}
dr(t)=[a_r^\mathbb{Q}(\gamma^\mathbb{Q}_r(t)-r(t))+b_r^{\mathbb{Q}}r(t-\tau)]dt+\sigma_r\sqrt{r(t)}dW^\mathbb{Q}_r(t),\\
r(t_0)=r_0(t)\quad t_0-\tau\leq t\leq t_0,\end{cases}
\end{equation}
where the $\mathbb{Q}$-parameters
\begin{equation}\label{eq:Qparameter_r-BP}a^\mathbb{Q}_r=a_r+\psi^r,\quad
\gamma^\mathbb{Q}_r(t)=\frac{a_r}{a_r+\psi^r}
\,\gamma_r(t), \quad b_r^{\mathbb{Q}}=b_r
\end{equation}
  are positive, the function $\gamma^\mathbb{Q}_r(t)$ is measurable, positive, and bounded on bounded intervals, moreover the Feller condition~\eqref{eq:feller's_condition_delay} under $\mathbb{P}$ automatically implies~\eqref{eq:feller's_condition_delay-Q}, the Feller condition under $\mathbb{Q}$. Indeed, with the above positions the market price of risk  $\xi_r(t)$~in~\eqref{eq:MarketPriceOfRisk-OneDim} coincides with the process $\xi_r(t,r(\cdot))$~in~\eqref{eq:xi-r-t-r}, and Theorem~\ref{thm:existenceQ} applies.
\\
As far  as Assumption~\ref{ass:MarketPriceOfRisk-OneDim} is concerned, actually, we could define the market price of risk also as follows
$$
\xi_r(t)=\psi_0^r\frac{\sqrt{r(t)}}{\sigma_r}+\psi_1^r\frac{1}{\sigma_r\sqrt{r(t)}}+\psi_2^r\frac{r(t-\tau)}{\sigma_r\sqrt{r(t)}},
$$
where $\psi_0^r$, $\psi_1^r$ and $\psi_2^r$ are constants satisfying  suitable conditions:
indeed, if  the $\mathbb{Q}$-parameters
\begin{align}\label{eq:Q-parameters-Gen}
a_r^\mathbb{Q}=a_r+\psi_0^r,\quad\gamma^\mathbb{Q}_r(t)=\frac{a_r\gamma_r(t)-\psi_1^r}{a_r+\psi_0^r},\quad b^\mathbb{Q}_r=b_r-\psi_2^r,
\end{align}
are positive and satisfy the Feller condition~\eqref{eq:feller's_condition_delay-Q}, then Theorem~\ref{thm:existenceQ} guarantees that the measure~$\mathbb{Q}$ is a probability measure, and that, under $\mathbb{Q}$, the process $r(t)$ has stochastic differential given by~\eqref{eq:IRQequation}.
As already observed, with our choice, i.e., with $\psi_0^r=\psi^r\geq 0$, $\psi_1^r=\psi_2^r=0$,  the above conditions are automatically satisfied, while, in general, this is not the case.
\end{oss}

In order to get the bond price for a fixed maturity $T$, the idea is to get a representation  of the following functional
(slightly more general than the functional~\eqref{eq:BondPrice} )
\begin{equation}\label{eq:FKformula}
\mathbb{E}^\mathbb{Q}\left[\mathit{e}^{-\int_t^Tr
(u)du-wr(T)}\bigg{|}\mathcal{F}_t
\right],\quad t\in[t_0,T], \text{ with } T \text{ fixed, and $w\geq 0$},
\end{equation}
as a deterministic function  $v^\mathbb{Q}(t,T,r,y;w)$  evaluated in $(r,y)=(r(t),y^\mathbb{Q}(t,T;w))$, where the  process $y^\mathbb{Q}(t,T;w)$ is  defined as follows
\begin{equation}\label{eq:def_yQ}
y^\mathbb{Q}(t,T;w):=\int_{t-\tau}^t \Gamma^\mathbb{Q}(u,T;w) r(u)\mathbf{1}_{[t_0-\tau,T-\tau]}(u) du,
\end{equation}
with $\Gamma^\mathbb{Q}(t,T;w)$ a suitably chosen deterministic function (for its explicit definition see \eqref{eq:def_GammaQ}).
Note that, independently of the definition of $\Gamma^\mathbb{Q}(t,T;w)$, the following final condition holds
\begin{equation}\label{eq:def_yQ(T,T)}
y^\mathbb{Q}(T,T;w)=0.
\end{equation}
\\
It turns out that the function $v^\mathbb{Q}$ is given by
\begin{equation}\label{eq:term_structure_one-dimensional_CIR_delay}
v^\mathbb{Q}(t,T,r,y;w)=\begin{cases}
\begin{array}{cc}
\mathit{e}^{-\alpha_0^\mathbb{Q}(t,T;w)-\alpha_r^\mathbb{Q}(t,T;w)r -y}& \text{for $t<T$},\\ \mathit{e}^{-wr-y}&\text{for $t=T$},
\end{array}\end{cases}\end{equation}
where $w$ is a nonnegative parameter and the functions $\alpha_0^\mathbb{Q}(t,T;w)$ and $\alpha_r^\mathbb{Q}(t,T;w)$ are deterministic and positive, and such that $\alpha_0^\mathbb{Q}(T,T;w)=0$ and $\alpha_r^\mathbb{Q}(T,T;w)=w$. \\

In the following Theorem \ref{thm:term_structure_one-dimensional_CIR_delay}, we give the correct choice of the functions  $\Gamma^\mathbb{Q}(t,T;w)$, $\alpha_0^\mathbb{Q}(t,T;w)$ and $\alpha_r^\mathbb{Q}(t,T;w)$.
Then the Bond price is obtained by \eqref{eq:FKformula}, with $w=0$, i.e.,
\begin{equation}\label{eq:BOND}
B(t,T)=v^\mathbb{Q}\big(t,T,r(t),y^\mathbb{Q}(t,T;0);0\big)=\mathit{e}^{-\alpha_0^\mathbb{Q}(t,T;0)
-\alpha_r^\mathbb{Q}(t,T;0)r(t) -y^\mathbb{Q}(t,T;0)},
\end{equation}
and we recover  $B(T,T)=1$ by the final condition in~\eqref{eq:term_structure_one-dimensional_CIR_delay}, and by~\eqref{eq:def_yQ(T,T)}.\\

\begin{thm}\label{thm:term_structure_one-dimensional_CIR_delay}${}$\\
With the notations and under the assumptions of Theorem~\ref{thm:existenceQ},
consider the following differential system
\begin{equation}\label{eq:term_structure_diff_system_delay-BP}\begin{cases}
\frac{d}{d t}\alpha_r (t)=&\frac{1}{2}\sigma_r^2(\alpha_r(t))^2+
a^\mathbb{Q}_r\alpha_r(t)-1\phantom{xxxxxxxxxxxx}\text{for $T-\tau\leq t\leq T$},\\
\vspace{2mm}\\
\frac{d}{dt}\alpha_r(t)=&\frac{1}{2}\sigma_r^2(\alpha_r(t))^2+a^\mathbb{Q}_r
\alpha_r(t)-1-b_r\alpha_r(t+\tau)\phantom{x}\text{for $t_0\leq t\leq T-\tau$},\\
\vspace{2mm}\\
\frac{d}{dt}\alpha_0 (t )=&-a^\mathbb{Q}_r\gamma^\mathbb{Q}_r(t)\alpha_r (t)
\phantom{xxxxxxxxxxxxxxxxxxxx}\text{for
$t_0\leq t\leq T$},\end{cases}\end{equation}
with the boundary conditions
\begin{equation}\label{eq:boundary_con_system_delay}\begin{cases}
\alpha_r(T)&=w,\\
\vspace{-2mm}\\
\alpha_r(T-\tau)&=\alpha_r((T-\tau)^+),\\
\vspace{-2mm}\\
\alpha_0(T)&=0.\end{cases}\end{equation}
Then, for all $w \in \left[0,\frac{\sqrt{(a_r^\mathbb{Q})^2 +2\sigma_r^2}-a_r^\mathbb{Q}}{\sigma_r^2}\right)$,
\begin{enumerate}
  \item the system \eqref{eq:term_structure_diff_system_delay-BP}-\eqref{eq:boundary_con_system_delay} has a unique solution $\left(\alpha_r^\mathbb{Q}(t,T;w);\alpha_0^\mathbb{Q}(t,T;w)\right)$.
  \item Moreover, the functions $\alpha_r^\mathbb{Q}(t,T;w)$ and $\alpha_0^\mathbb{Q}(t,T;w)$ are continuous, positive and right differentiable w.r.t.\@ $w$.\end{enumerate}
If  furthermore Assumption~\ref{ass:integrabilityUnidimensional} holds, and  the deterministic function $\Gamma^\mathbb{Q}(t,T;w)$ is chosen as follows\footnote{Actually, the function $\Gamma^\mathbb{Q}(t,T;w)$ can assume whatever value for $T-\tau\leq t\leq T$. The choice in \eqref{eq:def_GammaQ} has been made in order to make $\Gamma^\mathbb{Q}(t,T;w)$ a continuous function.}
\begin{equation}\label{eq:def_GammaQ}
\Gamma^\mathbb{Q}(t,T;w)=\begin{cases}\begin{array}{cc}
b_r\alpha^\mathbb{Q}_r(t+\tau,T;w)&\text{for $t_0\leq t\leq T-\tau$,}\ff\\
b_rw&\text{for $T-\tau\leq t\leq T$},\ff
\end{array}\end{cases}\end{equation}
then the generalized term structure
for the one-dimensional fixed delay CIR model  is given by the fun\-ction~$v^\mathbb{Q}$ (defined in \eqref{eq:term_structure_one-dimensional_CIR_delay}), i.e.,
\begin{equation}\label{eq:term_str}\mathbb{E}^\mathbb{Q}\left[\mathit{e}^{-\int_t^Tr(u)du-wr(T)}\bigg{|}\mathcal{F}_t
\right]=v^\mathbb{Q}(t,T,r(t),y^\mathbb{Q}(t,T;w);w),\end{equation}
where, for $t_0\leq t\leq T$,
\begin{align}
y^\mathbb{Q}(t,T;w)=\int_{t-\tau}^tb_r\alpha_r^\mathbb{Q}(u+\tau,T;w)r(u)
\mathbf{1}_{[t_0-\tau,t-\tau]}(u)du.\label{eq:def_yQ_w-BP}
\end{align}

\end{thm}

Before giving the proof, we make some observations.

\begin{oss}\label{oss:A1-A2-A3}${}$\\
  In the framework of Remark~\ref{oss:MeasureQ}, if Assumptions~\ref{asss:hp_r} and \ref{ass:integrabilityUnidimensional} hold, together with the Feller condition~\eqref{eq:feller's_condition_delay} under $\mathbb{P}$,
 the hypotheses of Theorem~\ref{thm:existenceQ} hold   under the further condition that $a^\mathbb{Q}_r$, $b^\mathbb{Q}_r$ and~$\gamma^\mathbb{Q}_r(t)$   in~\eqref{eq:Q-parameters-Gen}  be positive and satisfy the Feller condition~\eqref{eq:feller's_condition_delay-Q}, and therefore, replacing~$b_r$ with~$b^\mathbb{Q}_r$, Theorem~\ref{thm:term_structure_one-dimensional_CIR_delay} can be applied. In particular Theorem~\ref{thm:term_structure_one-dimensional_CIR_delay} can be applied
 if Assumptions~\ref{asss:hp_r}, \ref{ass:integrabilityUnidimensional}, and \ref{ass:MarketPriceOfRisk-OneDim} hold, together with the Feller condition~\eqref{eq:feller's_condition_delay} under $\mathbb{P}$, with  $a^\mathbb{Q}_r$, $b^\mathbb{Q}_r$ and $\gamma^\mathbb{Q}_r(t)$  as in~\eqref{eq:Qparameter_r-BP} without any further assumption.
\end{oss}

In accordance to \eqref{eq:YieldToMaturity}, by \eqref{eq:BOND} and \eqref{eq:term_str}, the term structure is a linear function of $r(t)$ and of the process $y^\mathbb{Q}(t,T;0)$:
\begin{equation}\label{eq:YtM-CIRdelay}
R(t,T)=\frac{1}{T-t}\left[\alpha_0^\mathbb{Q}(t,T;0)+
\alpha_r^\mathbb{Q}(t,T;0)r(t) +\int_{t-\tau}^t b_r\, \alpha_r^\mathbb{Q}(u+\tau,T;0)\, r(u)\mathbf{1}_{[t_0-\tau,T-\tau]}(u) du\right].
\end{equation}
The previous formula extends the formula of the term structure in the classical CIR model in which the rate $R(t,T)$ is affine function of $r(t)$.
    \begin{proof}[Proof of Theorem~\ref{thm:term_structure_one-dimensional_CIR_delay}]${}$\\
    We start with some preliminary observations.
Let $r(t)$ and $y^\mathbb{Q}(t,T;w)$ be the stochastic processes with dynamics described by \eqref{eq:IRQequation} and \eqref{eq:def_yQ}, with $\Gamma^\mathbb{Q}(t,T;w)$ still to be chosen.\\
The process $y^\mathbb{Q}(t,T;w)$ has stochastic differential given by
\begin{equation}\label{eq:def_Diff_y-CIR}\begin{split}
dy^\mathbb{Q}(t,T;w)=\, &\Gamma^\mathbb{Q}(t,T;w) r(t)\mathbf{1}_{[t_0-\tau,T-\tau]}(t)dt - \Gamma^\mathbb{Q}(t-\tau,T;w)\, r(t-\tau) \mathbf{1}_{[t_0-\tau,T-\tau]}(t-\tau)dt\\
=\, &\Gamma^\mathbb{Q}(t,T;w) r(t)\mathbf{1}_{[t_0-\tau,T-\tau]}(t)dt - \Gamma^\mathbb{Q}(t-\tau,T;w)\, r(t-\tau) \mathbf{1}_{[t_0,T]}(t)dt,\FF\end{split}
\end{equation}
and, by construction, the process $y^\mathbb{Q}(t,T;w)$, evaluated in $t=T$, is zero; actually,
\[y^\mathbb{Q}(T,T;w)=\int_{T-\tau}^T \Gamma^\mathbb{Q}(u,T;w) r(u)\mathbf{1}_{[t_0-\tau,T-\tau]}(u) du=0,\]
for any choice of $\Gamma^\mathbb{Q}(t,T;w)$.\\
Define the process $z(t)$ as follows
\begin{equation}\label{eq:z(t)_BP}z(t):=\mathit{e}^{-\int_{t_0}^t r(u)du} v^\mathbb{Q}(t,T,r(t),y^\mathbb{Q}(t,T;w);w),\end{equation}
where $v^\mathbb{Q}(t,T, r,y;w)$ is defined in~\eqref{eq:term_structure_one-dimensional_CIR_delay}, with $\alpha^\mathbb{Q}_0(t,T;w)$ and $\alpha^\mathbb{Q}_r(t,T;w)$ nonnegative and continuous in $t$, and such that $\alpha_0^\mathbb{Q}(T,T;w)=0$ and $\alpha_r^\mathbb{Q}(T,T;w)=w$, i.e., satisfy the boundary conditions~\eqref{eq:boundary_con_system_delay}. The idea is to show that the process $z(t)$ is a $\mathbb{Q}$-martingale if the functions $\alpha_0^\mathbb{Q}(t,T;w)$ and $\alpha_r^\mathbb{Q}(t,T;w)$ satisfy the system \eqref{eq:term_structure_diff_system_delay-BP}-\eqref{eq:boundary_con_system_delay} and $\Gamma^\mathbb{Q}(t,T;w)$ is defined as in~\eqref{eq:def_GammaQ}.
  Indeed, if $z(t)$ is a martingale, taking into account that $y^\mathbb{Q}(T,T;w)=0$,  and that therefore
\[z(T)=\mathit{e}^{-\int_{t_0}^Tr(u)du-wr(T)-y^\mathbb{Q}(T,T;w)}=\mathit{e}^{-\int_{t_0}^Tr(u)du-wr(T)},\]
we get the result, observing that
\[z(t)=\mathbb{E}^\mathbb{Q}\left[z(T)\big{|}\mathcal{F}_t\right]=\mathbb{E}^\mathbb{Q}\left[\mathit{e}^{-\int_{t_0}^T
r(u)du-wr(T)}\bigg{|}\mathcal{F}_t\right]\quad t_0\leq t\leq T,\]
and that, by the definition \eqref{eq:z(t)_BP} of $z(t)$, we have
\[\mathit{e}^{ -\int_{t_0}^t r(u)du} v^\mathbb{Q}\left(t,T,r(t), y^\mathbb{Q}(t,T;w);w\right)=\mathbb{E}^\mathbb{Q}\left[ \mathit{e}^{ -\int_{t_0}^Tr(u)du-wr(T)}\Big| \mathcal{F}_{t}\right],\]
that is
\[v^\mathbb{Q}\left(t,T,r(t), y^\mathbb{Q}(t,T;w);w\right)=\mathbb{E}^\mathbb{Q}\left[ \mathit{e}^{ -\int_{t}^Tr(u) du-wr(T)}\Big| \mathcal{F}_{t}\right].\]
The rest of the proof is devoted to show the martingale property of $z(t)$. To this end, an important observation is that, under Assumption~\ref{ass:integrabilityUnidimensional}, recalling Remark~\ref{oss:MeasureQ}, the process $r(t)$ is integrable w.r.t.\@ $\mathbb{Q}$, as immediately follows by Proposition~\ref{prop:integrabilityCIR} with $\mathbb{Q}$ instead of $\mathbb{P}$.\\

The process $z(t)$ (defined in \eqref{eq:z(t)_BP}) has stochastic differential given by
\begin{align}\label{eq:differenziale_z(t)_BP}
\nonumber dz(t)=&d\left(\mathit{e}^{-\int_{t_0}^t r(u)du} \right) v^\mathbb{Q}\left(t,T,r(t), y^\mathbb{Q}(t,T;w);w\right)
\\
 &+\mathit{e}^{-\int_{t_0}^tr(u)du}  dv^\mathbb{Q}\left(t,T,r(t), y^\mathbb{Q}(t,T;w);w\right).
\end{align}
By It\^{o}'s formula we obtain, for $t_0\leq t\leq T$,
\begin{align*}
dz(t)=&- r(t) z(t) dt + z(t)
\left[-\left(\tfrac{\partial }{\partial t}\alpha^\mathbb{Q}_0(t,T;w)+ r(t) \tfrac{\partial }{\partial t} \alpha^\mathbb{Q}_r(t,T;w)   \right) dt \right.\ff\\
&\left.-\alpha^\mathbb{Q}_r(t,T;w) dr(t)- dy^\mathbb{Q}(t,T;w) +  \tfrac{1}{2} \sigma_r^2 r(t) (\alpha_r^\mathbb{Q}(t,T;w))^2 dt \right]\FF
\\=&- r(t)z(t) dt + z(t)
\left[-\left(\tfrac{\partial }{\partial t}\alpha^\mathbb{Q}_0(t,T;w)+ r(t) \tfrac{\partial }{\partial t} \alpha^\mathbb{Q}_r(t,T;w)   \right) dt\right.\ff
\\
&\left.-\alpha^\mathbb{Q}_r(t,T;w)
 \left[a^\mathbb{Q}_r(\gamma^\mathbb{Q}_r(t)- r(t))+b_r r(t-\tau)\right]dt
-\alpha^\mathbb{Q}_r(t,T;w) \sigma_r\,\sqrt{|r(t)|}dW^\mathbb{Q}_r(t)\right.\ff
\\
&\left.-\Gamma^\mathbb{Q}(t,T;w)r(t)\mathbf{1}_{[t_0-\tau,T-\tau]}(t)dt
+\Gamma^\mathbb{Q}(t-\tau,T;w)r(t-\tau)\mathbf{1}_{[t_0,T]}(t)dt\right.\ff
\\
&\left.+  \tfrac{1}{2} \sigma_r^2 \, r(t) (\alpha_r^\mathbb{Q}(t,T;w))^2 dt  \right].\ff
\end{align*}
If $\Gamma^\mathbb{Q}$ is chosen as in \eqref{eq:def_GammaQ}, all the terms multiplying $r(t-\tau)$ disappear.\\
Then, the process $z(t)$ is a local martingale
if and only if the finite variation term vanishes, i.e., if and only if, for $t_0\leq t\leq T$,
\begin{equation}\label{eq:drift-z=0}
\begin{split}
{} &-\tfrac{\partial }{\partial t}\alpha^\mathbb{Q}_0(t,T;w)-\tfrac{\partial }{\partial t}\alpha^\mathbb{Q}_r(t,T;w) r(t)-r(t)-\alpha^\mathbb{Q}_r(t,T;w)[a^\mathbb{Q}_r(\gamma^\mathbb{Q}_r(t)- r(t))]\ff
\\
{} &- b_r\alpha^\mathbb{Q}_r(t+\tau,T;w) r(t)\mathbf{1}_{[t_0,T-\tau]}(t)
+\tfrac{1}{2}r(t)\sigma_r^2(\alpha_r^\mathbb{Q}(t,s;w))^2=0.
\end{split}
\end{equation}
Moreover, thanks to the previous observation on the integrability of $r(t)$,
the process $z(t)$ is a (square integrable) martingale if $\alpha^\mathbb{Q}_0(u,T;w)$ and $\alpha^\mathbb{Q}_r(u,T;w)$ are nonnegative continuous functions; indeed, then $0\leq z(t)\leq 1$, and  setting $m_r(t):=\sigma_r z(t)\sqrt{|r(t)|}\alpha^\mathbb{Q}_r(t,T;w)$, we have that
\begin{align*}
\mathbb{E}^\mathbb{Q}\left[\int_{t_0}^T|m_r(t)|^2dt\right]
\leq&\sigma_r^2\max_{t_0\leq u\leq T}|\alpha^\mathbb{Q}_r(u,T;w)|^2\int_{t_0}^T\mathbb{E}^\mathbb{Q}\left[|r(t)|\right]dt<+\infty.
\end{align*}
Gathering in \eqref{eq:drift-z=0} the terms multiplying $r(t)$, we get the condition
\begin{align*}
&\left(-\tfrac{\partial}{\partial t}\alpha^\mathbb{Q}_r(t,T;w) +a^\mathbb{Q}_r\alpha^\mathbb{Q}_r (t,T;w)-b_r\alpha^\mathbb{Q}_r (t+\tau,T;w)\mathbf{1}_{[t_0,T-\tau]}(t)+\tfrac{1}{2}\sigma^2_r(\alpha^\mathbb{Q}_r(t,T;w))^2-1\right)r(t)\ff
\\
&+\left(-\tfrac{\partial}{\partial t}\alpha^\mathbb{Q}_0(t,T;w)-a^\mathbb{Q}_r\gamma^\mathbb{Q}_r(t)
\alpha^\mathbb{Q}_r(t,T;w)\right)=0.
\end{align*}
Since the previous equation holds for all $r(t)\geq0$, the functions $\alpha^\mathbb{Q}_r(t,T;w)$  and  $\alpha^\mathbb{Q}_0(t,T;w)$ solve the system~\eqref{eq:term_structure_diff_system_delay-BP} with the respective boundary conditions~\eqref{eq:boundary_con_system_delay}.\\

By Lemma~\ref{lem:LemmaTecnico} (see Appendix~\ref{app:ProofsCIR}), with $a=a^\mathbb{Q}_r$, $b=b_r$, and $\sigma=\sigma_r$, the ordinary
 differential equation
\[\begin{cases}
-\tfrac{d}{d t}\alpha_r(t) +a^\mathbb{Q}_r\alpha_r(t)-b_r\alpha_r(t+\tau)\mathbf{1}_{[t_0,T-\tau]}(t)+\tfrac{1}{2}\sigma^2_r\alpha_r^2(t)-1=0,\FF\\
\alpha_r(T)=w,\FF
\end{cases}\]
has a unique  solution $\alpha^\mathbb{Q}_r(t,T;w)$,    positive and right differentiable w.r.t\@ $w$.\\

Consequently, also the following ordinary
 differential equation
\[\begin{cases}
&-\tfrac{d}{d t}\alpha_0(t)-a^\mathbb{Q}_r\gamma^\mathbb{Q}_r(t)
\alpha_r(t)=0\quad\text{for $t_0\leq t \leq T$},\\
&\alpha_0(T)=0,
\end{cases}\]
has a unique solution $\alpha^\mathbb{Q}_0(t,T;w)$, given by
\[\alpha^\mathbb{Q}_0(t,T;w)=a^\mathbb{Q}_r\int_t^T\gamma^\mathbb{Q}_r(u)\alpha^\mathbb{Q}_r(u,T;w)du,\]
positive and right differentiable w.r.t\@ $w$.

    \end{proof}

\section{Instantaneous Forward Rate}\label{sec:FR}
Since very often the traders are interested to determine the future yield on a bond, given by  the instantaneous forward rate $f(t,T)$, we focus our interest on it.

The main result of this section states that if the spot rate is a fixed delay CIR process    $r(t)$, and if the Assumptions~\ref{asss:hp_r}, \ref{ass:integrabilityUnidimensional}, and \ref{ass:MarketPriceOfRisk-OneDim} hold, then the instantaneous forward rate is a deterministic linear function of the process $r(t)$ and another suitable process $\tilde{y}^\mathbb{Q}(t,T;0)$; that is
\begin{equation}\label{eq:FRclosedformula}
f(t,T):=\beta_0^\mathbb{Q}(t,T;0)+\beta_r^\mathbb{Q}(t,T;0)r(t)+\tilde{y}^\mathbb{Q}(t,T;0),
\end{equation}
where $\beta_0^\mathbb{Q}(t,T;0)$ and $\beta_r^\mathbb{Q}(t,T;0)$ are deterministic functions (see Theorem~\ref{thm:term_structure_derivate_delay-one-dim} and the subsequent Remark~\ref{rem:linear-repr-FR}).
 Thus we obtain a generalization of the well-known property of the classical CIR model (\cite{CIR:SIR}).
 \\
 More precisely   $\tilde{y}^\mathbb{Q}(t,T;0)$, $\beta_0^\mathbb{Q}(t,T;0)$ and $\beta_r^\mathbb{Q}(t,T;0)$ are obtained by taking the partial derivatives in $w=0$ of  $y^\mathbb{Q}(t,T;w)$, $\alpha_0^\mathbb{Q}(t,T;w)$ and $\alpha_r^\mathbb{Q}(t,T;w)$, respectively
 (see \eqref{eq:def_tildeyQ}  and \eqref{eq:beta_alpha}). In Theorem~\ref{thm:term_structure_derivate_delay-one-dim} we show that $\beta_0^\mathbb{Q}(t,T;0)$ and $\beta_r^\mathbb{Q}(t,T;0)$ are characterized as the solution of a deterministic linear system of differential equations, and finally the  process $\tilde{y}^\mathbb{Q}(t,T;0)$ has an alternative expression (see \eqref{eq:def_tildeyQ_BP}).
 \\
We start by recalling  the definition and some properties of the forward rate.
Let $f(t,T,S)$ be the forward rate at time $t$ for the expiry time $T$ and maturity time $S$. In an Arbitrage-free market, the following equality holds
$$
\mathit{e}^{R(t,S)(S-t)}=\mathit{e}^{R(t,T)(T-t)}\mathit{e}^{f(t,T,S)(S-T)},
$$
so that
\begin{equation}\label{eq:ForwardRate}
f(t,T,S):= -\frac{\ln(B(t,S))-\ln(B(t,T))}{S-T}.
\end{equation}

\begin{defn}\label{defn:InstanteousForwardRate}${}$\\
The instantaneous forward rate  (or shortly forward rate) at time $t$ with maturity time $T>t$, $f(t,T)$ is defined by
\begin{equation}\label{eq:InstanteousForwardRate}
f(t,T)=\lim_{S\rightarrow T}f(t,T,S)=-\frac{\partial}{\partial T} \left(\log B(t,T)\right)=-\frac{1}{B(t,T)}\frac{\partial}{\partial T}B(t,T).
\end{equation}
\end{defn}
It corresponds to the instantaneous interest rate that one can contract at time $t$, on a risk-less loan that begins at the date $T$ and is returned on a date later than $T$.\\
By \eqref{eq:InstanteousForwardRate}, we can computed the price of a  uZCB as a functional of the instantaneous forward rate; that is,
\begin{equation}\label{eq:BPwithFR1}B(t,T)=\mathit{e}^{-\int_t^Tf(t,u)du}.\end{equation}

\begin{prop}\label{prop:FormulaTassoForward}${}$\\
Let the spot rate be a nonnegative process $r(t)$. Assume that the process $r(t)$ is integrable and uniformly in bounded intervals, then in order to ensure that this financial market satisfies the no-arbitrage condition, the following condition holds
\begin{equation}\label{eq:gen_forward_rate}f(t,T)=\frac{\mathbb{E}^{\mathbb{Q}}\left[r(T)\mathit{e}^
{-\int_t^Tr(u)du}\bigg{|}\mathcal{F}_t\right]}{\mathbb{E}^{\mathbb{Q}}\left[\mathit{e}^{-\int_t^Tr(u)du}
\bigg{|}\mathcal{F}_t\right]}.\end{equation}
Furthermore the numerator in the previous equation, can be evaluated as
\begin{equation}\label{eq:gen_forward_rateNUM}
\mathbb{E}^{\mathbb{Q}}\left[r(T)\mathit{e}^
{-\int_t^Tr(u)du}\bigg{|}\mathcal{F}_t\right]\, =\,-\,\frac{\partial}{\partial w^+}\,\mathbb{E}^\mathbb{Q}\left[\mathit{e}^{-\int_t^Tr(u)du-wr(T)}\bigg{|}\mathcal{F}_t\right]\bigg|_{w=0}.
\end{equation}

\end{prop}
\begin{proof}${}$\\
Taking into account that $f(t,T)$ is $\mathcal{F}_t$-measurable, the equality~\eqref{eq:gen_forward_rate} is equivalent to
\begin{equation}\label{eq:FR_equality}
\mathbb{E}^{\mathbb{Q}}\left[f(t,T)\mathit{e}^{-\int_t^Tr(u)du}\bigg{|}\mathcal{F}_t\right]
=\mathbb{E}^{\mathbb{Q}}\left[r(T)\mathit{e}^{-\int_t^Tr(u)du}\bigg{|}\mathcal{F}_t\right].
\end{equation}
Observing that, by  \eqref{eq:BPwithFR1} and \eqref{eq:BondPrice},
\[B(t,T)\,\frac{\mathit{e}^{-\int_T^{T+h}f(t,u)du}-1}{h}=\frac{B(t,T+h)-B(t,T)}{h}
=\mathbb{E}^{\mathbb{Q}}\left[\mathit{e}^{-\int_t^Tr(u)du}\,\frac{\mathit{e}^{-\int_T^{T+h}r(u)du}-1}{h}
\bigg{|}\mathcal{F}_t\right],\]
and letting $h\rightarrow0^+$, the left-hand side converges to $ B(t,T)f(t,T)$, and the right-hand side converges to $\mathbb{E}^{\mathbb{Q}}\left[r(T)\mathit{e}^{-\int_t^Tr(u)du}\bigg{|}\mathcal{F}_t\right]$. The latter limit holds thanks to the observation that, $r(t)$ being nonnegative,
$$
\left|\mathit{e}^{-\int_t^Tr(u)du}\,\frac{\mathit{e}^{-\int_T^{T+h}r(u)du}-1}{h}\right|\leq \sup_{T\leq u \leq T+1} r(u), \quad \text{for all} \, 0\leq h \leq 1,
$$
and the integrability condition on $r(t)$.

Similarly we get
\[\frac{\partial}{\partial w^+}\mathbb{E}^\mathbb{Q}\left[\mathit{e}^{-\int_t^Tr(u)du-wr(T)}\bigg{|}\mathcal{F}_t\right]=-
\mathbb{E}^\mathbb{Q}\left[r(T)\mathit{e}^{-\int_t^Tr(u)du-wr(T)}\bigg{|}\mathcal{F}_t\right],\]
and therefore, taking $w=0$, we get \eqref{eq:gen_forward_rateNUM}.
\end{proof}

\bigskip

To obtain formula \eqref{eq:FRclosedformula}, we need   a representation formula for the numerator of~\eqref{eq:gen_forward_rate}. In this regard, as we have seen in the previous section (see Theorem~\ref{thm:term_structure_one-dimensional_CIR_delay}), if the spot rate $r(t)$  is a   fixed delay CIR process, we can represent
\[\mathbb{E}^\mathbb{Q}\left[\mathit{e}^{-\int_t^Tr(u)du-wr(T)}\bigg{|}\mathcal{F}_t\right]=v^\mathbb{Q}
(t,T,r(t),y^\mathbb{Q}(t,T;w);w),\]
where the function $v^\mathbb{Q}(t,T,r,y;w)$ is defined in \eqref{eq:term_structure_one-dimensional_CIR_delay}.
  Then accordingly to~\eqref{eq:gen_forward_rateNUM} in Proposition~\ref{prop:FormulaTassoForward}, we can represent
\[\mathbb{E}^\mathbb{Q}\left[r(T)\mathit{e}^{-\int_t^Tr(u)du-wr(T)}\bigg{|}\mathcal{F}_t\right]\, =\,-\,\frac{\partial}{\partial w^+}v^\mathbb{Q}
(t,T,r(t),y^\mathbb{Q}(t,T;w);w)
\quad \text{$t\in[t_0,T]$ with $T$ fixed},\]
where $y^\mathbb{Q}(t,T;w)$ is the process defined in~\eqref{eq:def_yQ}.
\\
As we will prove below, the main observation is that the left-hand side of the previous equality can be expressed
as a function $\tilde{v}^\mathbb{Q}(t,T,r,y,\tilde{y};w)$ (see its expression in~\eqref{eq:tilde_vQ}), evaluated in $(r,y,\tilde{y})=(r(t),y^\mathbb{Q}(t,T;w),\tilde{y}^\mathbb{Q}(t,T;w))$,
where $\tilde{y}^\mathbb{Q}(t,T;w)$ is
\begin{align}\notag
\tilde{y}^\mathbb{Q}(t,T;w)&=\frac{\partial}{\partial w^+}y^\mathbb{Q}(t,T;w)=
\int_{t-\tau}^t\frac{\partial}{\partial w^+}\Gamma^\mathbb{Q}(u,T;w)r(u)\mathbf{1}_{[t_0-\tau,t-\tau]}(u)du
\intertext{(thanks to the expression~\eqref{eq:def_GammaQ} of $\Gamma^\mathbb{Q}(t,T;w)$)}
&=\int_{t-\tau}^tb_r \frac{\partial}{\partial w^+}\alpha^\mathbb{Q}_r(u+\tau,T;w)r(u)
\mathbf{1}_{[t_0-\tau,t-\tau]}(u)du,
\label{eq:def_tildeyQ}
\end{align}
The function $\tilde{v}^\mathbb{Q}$ is given by
\begin{equation}\label{eq:tilde_vQ}
\tilde{v}^\mathbb{Q}(t,T,r,y,\tilde{y};w)=\begin{cases}\begin{array}{cc}
\left(\beta^\mathbb{Q}_0(t,T;w)+r\beta^\mathbb{Q}_r(t,T;w)+\tilde{y}\right)\mathit{e}^{-\alpha_0^\mathbb{Q}(t,T;w)
-\alpha_r^\mathbb{Q}(t,T;w)r -y}& \text{$t<T$}\\
(r+\tilde{y})\mathit{e}^{-wr-y}&\text{$t=T$},\end{array}\end{cases}
\end{equation}
where we have set
\begin{align}
\beta^\mathbb{Q}_0(t,T;w)=\frac{\partial}{\partial w^+}\alpha^\mathbb{Q}_0(t,T;w),
\qquad
\beta^\mathbb{Q}_r(t,T;w)=\frac{\partial}{\partial w^+}\alpha^\mathbb{Q}_r(t,T;w).\label{eq:beta_alpha}\end{align}
Indeed, for $t<T$
\begin{align*}
&\tilde{v}^\mathbb{Q}(t,T,r(t),y^\mathbb{Q}(t,T;w),\tilde{y}^\mathbb{Q}(t,T;w);w)=
-\tfrac{\partial}{\partial w^+}v^\mathbb{Q}(t,T,r(t),y^\mathbb{Q}(t,T;w);w)\\
&=v^\mathbb{Q}(t,T,r(t),y^\mathbb{Q}(t,T;w);w)\left(\tfrac{\partial}{\partial w^+}\alpha^\mathbb{Q}_0(t,T;w)+r(t)\tfrac{\partial}{\partial w^+}\alpha^\mathbb{Q}(t,T;w)+
\tfrac{\partial}{\partial w^+}y^\mathbb{Q}(t,T;w)\right)\\
&=v^\mathbb{Q}(t,T,r(t),y^\mathbb{Q}(t,T;w);w)\left(\beta^\mathbb{Q}_0(t,T;w)+r(t)\beta^\mathbb{Q}(t,T;w)
+\tilde{y}^\mathbb{Q}(t,T;w)\right),
\end{align*} while for $t=T$
\begin{align*}
\tilde{v}^\mathbb{Q}(T,T,r(T),y^\mathbb{Q}(T,T;w),\tilde{y}^\mathbb{Q}(T,T;w);w)=
&-\tfrac{\partial}{\partial w^+}v^\mathbb{Q}(T,T,r(T),y^\mathbb{Q}(T,T;w);w)\\
=&v^\mathbb{Q}(T,T,r(T),y^\mathbb{Q}(T,T;w);w)(r(T)+\tilde{y}^\mathbb{Q}(T,T;w)).
\end{align*}
Observe that,
since $y^\mathbb{Q}(T,T;w)=0$ and $\tilde{y}^\mathbb{Q}(T,T;w)=0$ for all $w$,  the latter formula coincides with $v^\mathbb{Q}(T,T,r(T),0;w)r(T)$, from which one can reobtain the obvious identity $f(T,T)=r(T)$. \\

With the following theorem, we characterize the functions $\beta_0^\mathbb{Q}(t,T;w)$ and $\beta_r^\mathbb{Q}(t,T;w)$ as the solutions of a system of linear differential equations.
\begin{thm}\label{thm:term_structure_derivate_delay-one-dim}${}$\\
Let the risk-free interest rate $r(t)$ be the process described by \eqref{eq:IRQequation}, under the probability measure~$\mathbb{Q}$. Let $\alpha_0^\mathbb{Q}(t,T;w)$, $\alpha_r^\mathbb{Q}(t,T;w)$ be the continuous solution of system  \eqref{eq:term_structure_diff_system_delay-BP}-\eqref{eq:boundary_con_system_delay}. Assume that the deterministic function $\Gamma^\mathbb{Q}(t,T;w)$ is chosen as  in \eqref{eq:def_GammaQ}.
Then, under Assumptions~\ref{asss:hp_r}, \ref{ass:integrabilityUnidimensional} and \ref{ass:MarketPriceOfRisk-OneDim}, we have that, for all $w \in \left[0,\frac{\sqrt{(a_r^\mathbb{Q})^2 +2\sigma_r^2}-a_r^\mathbb{Q}}{\sigma_r^2}\right)$,
\begin{enumerate}
\item the following linear differential system
\begin{equation}\label{eq:term_struction_derivate_system_delay-one-dim}\begin{cases}
\frac{d}{dt}\beta_r(t)=&\left(\sigma_r^2\alpha_r^\mathbb{Q}(t,T;w)+a^\mathbb{Q}_r\right)
\beta_r(t)
\phantom{xxxxxxxxxxxxx}\text{for $T-\tau\leq t\leq T$}\\
\vspace{2mm}\\
\frac{d}{dt}\beta_r(t)=&\left(\sigma_r^2\alpha^\mathbb{Q}_r(t,T;w)+a^\mathbb{Q}_r\right)
\beta_r(t)-b_r\beta_r(t+\tau)\phantom{xxx}\text{for $t_0\leq t\leq T-\tau$}\\
\vspace{2mm}\\
\frac{d}{dt}\beta_0(t)=&-a^\mathbb{Q}_r\gamma^\mathbb{Q}_r(t)\beta_r(t)
\phantom{xxxxxxxxxxxxxxxxxxxxxx}\text{for $t_0\leq t\leq  T$}
\end{cases}
\end{equation}
with the boundary conditions
\begin{equation}\label{eq:BC_derivative_system}\begin{cases}
\beta_r(T)&=1,\\
\beta_r(T-\tau)&=\beta_r((T-\tau)^+),\\
\beta_0(T)&=0,
\end{cases}\end{equation}
has a unique solution with components $\beta^\mathbb{Q}_r(t,T;w)$ and $\beta^\mathbb{Q}_0(t,T;w)$, coinciding  with the functions defined in~\eqref{eq:beta_alpha};
\item the functions $\beta^\mathbb{Q}_r(t,T;w)$ and $\beta^\mathbb{Q}_0(t,T;w)$ are continuous and positive;
\item the following representation formula holds:
\begin{equation}\label{eq:num_FRw}\begin{split}
&\mathbb{E}^\mathbb{Q}\left[r(T)\mathit{e}^{-\int_t^Tr(u)du-wr(T)}\left|\mathcal{F}_t\right.\right]=
\tilde{v}^\mathbb{Q}(t,T,r(t),y^\mathbb{Q}(t,T;w), \tilde{y}^\mathbb{Q}(t,T;w);w)
\\&=
\left(\beta^\mathbb{Q}_0(t,T;w)+r(t)\beta^\mathbb{Q}_r(t,T;w)+\tilde{y}^\mathbb{Q}(t,T;w)\right)
\mathit{e}^{-\alpha_0^\mathbb{Q}(t,T;w)-\alpha_r^\mathbb{Q}(t,T;w)r(t)-y^\mathbb{Q}(t,T;w)},
\end{split}
\end{equation}
where, for $t_0\leq t\leq T$, $y^\mathbb{Q}(t,T;w)$ is given in \eqref{eq:def_yQ_w-BP}, and
\begin{align}
\tilde{y}^\mathbb{Q}(t,T;w)=\int_{t-\tau}^tb_r\beta_r^\mathbb{Q}(u+\tau,T;w)r(u)
\mathbf{1}_{[t_0-\tau,t-\tau]}(u)du.\label{eq:def_tildeyQ_BP}
\end{align}

\end{enumerate}
\end{thm}
\begin{oss}\label{rem:linear-repr-FR}
The announced linear representation~\eqref{eq:FRclosedformula} of the instantaneous forward rate   now can be easily derived.
Indeed, under the assumptions of Theorem~\ref{thm:term_structure_derivate_delay-one-dim},  \eqref{eq:FRclosedformula} follows by the definition~\eqref{eq:gen_forward_rate} of the instantaneous forward rate, together with \eqref{eq:num_FRw},
\eqref{eq:term_str}, and \eqref{eq:term_structure_one-dimensional_CIR_delay}, all  evaluated in $w=0$.

Furthermore, as a direct consequence of~\eqref{eq:BPwithFR1},
we can also represent the zero-coupon bond price with the following relation
\begin{equation}\label{eq:BPwithFR2}
B(t,T)=\mathit{e}^{-\int_t^T\left[\beta^\mathbb{Q}_0(t,u;0)+r(t)\beta^\mathbb{Q}_r(t,u;0)+
\tilde{y}^\mathbb{Q}(t,u;0)\right]du}.
\end{equation}

\end{oss}

\begin{proof}[Proof of Theorem~\ref{thm:term_structure_derivate_delay-one-dim}]
${}$\\
We prove only points $1.$ and $2.$ since thanks to \eqref{eq:tilde_vQ} and \eqref{eq:beta_alpha},
the point $3.$ immediately follows.\\
Right-differentiating with respect to the variable $w$, the first equation of system \eqref{eq:term_structure_diff_system_delay-BP}, we obtain for $T-\tau\leq t\leq T$
\begin{align*}\tfrac{\partial}{\partial w^+}\left(\tfrac{\partial}{\partial t}\alpha^\mathbb{Q}_r(t,T;w)\right)&=\sigma_r^2\alpha^\mathbb{Q}_r(t,T;w)\tfrac{\partial }{\partial w^+}\alpha^\mathbb{Q}_r(t,T;w)+a^\mathbb{Q}_r\tfrac{\partial}{\partial w^+} \alpha^\mathbb{Q}_r(t,T;w)\\&=\left(\sigma_r^2\,\alpha^\mathbb{Q}_r(t,T;w)+a^\mathbb{Q}_r\right)
\tfrac{\partial }{\partial w^+}\alpha^\mathbb{Q}_r(t,T;w).\end{align*}
Then, formally, by the first definition in \eqref{eq:beta_alpha}, we have
\[\tfrac{\partial }{\partial t}\beta^\mathbb{Q}_r(t,T;w)=\left(\sigma_r^2\,\alpha^\mathbb{Q}_r(t,T;w)+a^\mathbb{Q}_r\right)
\beta^\mathbb{Q}_r(t,T;w).\]
A rigorous proof of the above equation can be achieved   by standard results on ordinary differential equations, depending on a parameter, under global Lipschitz conditions, thanks to Remark~\ref{rem:LemmaTecnico} (see Appendix~\ref{app:ProofsCIR}).
Solving this equation with the boundary condition
\[\beta^\mathbb{Q}_r(T,T;w)=\tfrac{\partial }{\partial w^+}\alpha^\mathbb{Q}_r(T,T;w)=1,\]
we obtain that the unique solution is given by
\begin{equation}\label{eq:betaQ_r1}
\beta^\mathbb{Q}_r(t,T;w)=\mathit{e}^{-\int_t^T\left(\sigma_r^2\,\alpha^\mathbb{Q}_r(u,T;w)+
a^\mathbb{Q}_r\right)du},\quad\text{for $T-\tau\leq t \leq T$},
\end{equation}
which is positive.
Similarly we get
\begin{equation}\label{eq:betaQ_r2_edo}
\tfrac{\partial }{\partial t}\beta^\mathbb{Q}_r(t,T-\tau;w)=\left(\sigma_r^2\alpha^\mathbb{Q}_r(t,T;w)+a_r\right)\beta^\mathbb{Q}_r(t,T-\tau;w)
-b_r\beta^\mathbb{Q}_r(t+\tau,T;w)\quad\text{for $t_0\leq t\leq T-\tau$},
\end{equation}
 with the boundary condition given by the solution of \eqref{eq:betaQ_r1}, evaluated in $t=T-\tau$;
  the unique solution is given by
\begin{equation}\label{eq:betaQ_r2-}\begin{split}
\beta^\mathbb{Q}_r(t,T-\tau;w)=\beta^\mathbb{Q}_r(t,T;w)+\,b_r\int_t^{T-\tau}\mathit{e}^{-\int_t^s
\left(\sigma^2_r\alpha^\mathbb{Q}_r(u,T;w)+a^\mathbb{Q}_r\right)du}\beta^\mathbb{Q}_r(s+\tau,T;w)ds,
\end{split}\end{equation}
and is positive.
The same procedure applies to  the third equation of the system \eqref{eq:term_structure_diff_system_delay-BP}, and recalling the definition~\eqref{eq:beta_alpha}, of $\beta^\mathbb{Q}_0(t,T;w)$,
we obtain the equation
\begin{align}\label{eq:betaQ_0_edo}
 \frac{\partial }{\partial t}\beta^\mathbb{Q}_0(t,T;w)&=-a^\mathbb{Q}_r\gamma^\mathbb{Q}_r(t)\beta^\mathbb{Q}_r(t,T;w)\quad\text{for $t_0\leq t\leq T$},
 \intertext{with boundary condition}
\beta^\mathbb{Q}_0(T,T;w)&=\tfrac{\partial }{\partial w^+}\alpha^\mathbb{Q}_0(T,T;w)=0.\notag
 \end{align}
The unique solution of Eq.~\eqref{eq:betaQ_0_edo} is positive and is given by
 \begin{equation}\label{eq:betaQ_0}
\beta^\mathbb{Q}_0(t,T;w)=a^\mathbb{Q}_r\int_t^T\gamma^\mathbb{Q}_r(u)\beta^\mathbb{Q}_r(u,T;w)du.
\end{equation}

\end{proof}

\appendix
\section{Appendix}\label{app:ProofsCIR}

\begin{proof}[\large{\textbf{\emph{Proof of Theorem \ref{thm:existenceQ} (see page \pageref{thm:existenceQ})}}}] ${}$\\
Since the $\mathbb{P}$-parameters satisfy the Feller condition, the solution the process $r(t)$ is positive for all $t>t_0$; consequently, the process $\xi_r(t,r(\cdot))$ in \eqref{eq:xi-r-t-r}
is a well-defined continuous process. Therefore, we can define the nonnegative supermartingale given by
$Z_t=1$, when $t\leq t_0$, and
$$
  Z_t  :=\exp\left\{ -\int_{t_0}^t \xi_r(s,r(\cdot)) \, dW^\mathbb{P}(s)-  \frac{1}{2}\,\int_{t_0}^t \xi^2_r(s,r(\cdot))\, ds\right\},\quad t\in[t_0,T].
$$

If $Z_t$ is a $\mathbb{P}$-martingale then, as usual, one can define the probability measure $\mathbb{Q}$ on $\mathcal{F}_T$, so that
  $$
  d \mathbb{Q} = Z_T \, d \mathbb{P}.
  $$
By Girsanov theorem
\begin{equation}\label{eq:W-Q}
W_r^\mathbb{Q}(t):= W^\mathbb{P}(t)+\int_{t_0}^t \xi_r(s,r(\cdot))\, ds, \quad t\in [t_0,T],
\end{equation}
is a Brownian motion, under $\mathbb{Q}$ and
  $$
  r(t)=r_0(t), \; t\in [t_0-\tau, t_0],\quad  r(t)= r_0(t_0)+ \int_{t_0}^t \mu^\mathbb{Q}(s,r(\cdot))ds+ \int_{t_0}^t \sigma_r
  \sqrt{r(s)} \, dW_r^\mathbb{Q}(s), \; t\in [t_0,T],
  $$
and the thesis is achieved.

The process $Z_t$ is a martingale if and only if
\begin{equation}\label{eq:E[Z-t]=1}
\mathbb{E}^{\mathbb{P}}[Z_T]=1.
\end{equation}

In order to prove~\eqref{eq:E[Z-t]=1},  we define the process $\widetilde{r}(t)$ on the probability space $(\Omega, \mathcal{F}, \{\mathcal{F}_t\}, \mathbb{P})$, as the strong solution of the following SDDE
  \begin{equation}\label{eq:r-tilde}
  \begin{split}
   \widetilde{r}(t)&=r_0(t_0)+ \int_{t_0}^t \mu^\mathbb{Q}(s,\widetilde{r}(s))ds+ \int_{t_0}^t \sigma_r
  \sqrt{\widetilde{r}(s)} \, dW_r^\mathbb{P}(s), \; t\in [t_0,T],
  \\
  \widetilde{r}(t)&=r_0(t), \; t\in [t_0-\tau, t_0],
 \end{split}
 \end{equation}
  By hypotheses, the distribution of $\widetilde{r}(t)$ is unique (see Remark~\ref{oss:UiL}). Since the $\mathbb{Q}$-parameters also satisfy
the Feller condition,  the following  process
  $$
  \xi_r(t,  \widetilde{r}(\cdot)):= \frac{\mu^\mathbb{P}(t,  \widetilde{r}(\cdot))-\mu^\mathbb{Q}(t,  \widetilde{r}(\cdot))}{\sigma_r\sqrt{  \widetilde{r} (t)}}= \frac{a_r\gamma_r(t) - a_r^\mathbb{Q}\gamma^\mathbb{Q}(t) -(a_r- a_r^\mathbb{Q})   \widetilde{r}(t) + (b_r-b_r^\mathbb{Q})  \widetilde{r} (t-\tau)}{\sigma_r\sqrt{  \widetilde{r} (t)}}
  $$
  is well defined and with continuous paths, as well as the process $\xi_r(t,  r(\cdot))$. Therefore, if we denote by
  $\tau_n$ and $\widetilde{\tau}_n$  the stopping times
  $$
  \tau_n(\omega)=\inf\{t>t_0:\; |\xi_r(t,r(\cdot))|\geq n \}\wedge T, \qquad \widetilde{\tau}_n(\omega)=\inf\{t>t_0:\; |\xi_r(t,\widetilde{r}(\cdot))|\geq n \}\wedge T,
  $$
  then we get
  $$
  \lim_{n\rightarrow \infty}\mathbb{P}(\tau_n=T)=\lim_{n\rightarrow \infty}\mathbb{P}(\widetilde{\tau}_n=T)=1.
  $$
  For any $n\geq 1$, we can define
  $$
  \xi_r^{(n)}(t):=\xi_r(t,r(\cdot)) \,\mathbf{1}_{\{t\leq \tau_n\}}
  $$
  so that
  $$
  \int_{t_0}^t |\xi_r^{(n)}(s)|^2\, ds\leq n^2 (t-t_0).
  $$
For each $n$, the process satisfies the Novikov condition
  $$
  \mathbb{E}^{\mathbb{P}}\left[ \exp\left\{\frac{1}{2}\,\int_{t_0}^T |\xi_r^{(n)}(s)|^2\, ds\right\}\right]\leq \exp\{\tfrac{n^2 (T-t_0)}{2}\} <\infty,
  $$
and, it follows that, for each $n\geq1$, the process defined by
 $$
  Z^{(n)}_t  :=\exp\left\{ -\int_{t_0}^t \xi_r^{(n)}(s) \, dW_r^\mathbb{P}(s)-  \frac{1}{2}\,\int_{t_0}^t |\xi_r^{(n)}(s)|^2\, ds\right\}
  $$
   is a $\mathbb{P}$-martingale such that $ \mathbb{E}^\mathbb{P}[Z^{(n)}_t]=1$ for all $t\leq T$. Consequently, we can define a probability measure  $\mathbb{Q}^{(n)}$, on $\mathcal{F}_T$ as follows
  $$
  d \mathbb{Q}^{(n)} = Z^{(n)}_T \, d \mathbb{P}.
  $$
Since $\mathbb{P}(\tau_n=T)\rightarrow 1$,
and the sequence $\tau_n $ is monotone increasing, we have that
$$
Z^{(n)}_T \,\mathbf{1}_{\{ \tau_n=T\}} = Z_T \,\mathbf{1}_{\{ \tau_n=T\}} \nearrow Z_T, \quad  \mathbb{P}-a.s.
$$
By Beppo-Levi's monotone convergence theorem  and the definition of $\mathbb{Q}^{(n)}$, we have that
\begin{equation}\label{eq:Q-n-tau=T}
\mathbb{E}^{\mathbb{P}}[Z_T]=\lim_{n \rightarrow \infty} \mathbb{E}^{\mathbb{P}}[Z^{(n)}_T\,\mathbf{1}_{\{ \tau_n=T\}}]
= \lim_{n \rightarrow \infty}\mathbb{E}^{\mathbb{Q}^{(n)}}[\,\mathbf{1}_{\{ \tau_n=T\}}]= \lim_{n \rightarrow \infty}\mathbb{Q}^{(n)} (\tau_n=T).
\end{equation}
 Moreover, by Girsanov theorem
  $$
  W_r^{(n)}(t):= W_r^\mathbb{P}(t)+\int_{t_0}^t \xi_r^{(n)}(s)\, ds, \quad t\in [t_0,T]
  $$
  is a Browniam motion under the measure $\mathbb{Q}^{(n)}$ and
  $$
 r(t)=r_0(t_0)+ \int_{t_0}^t \mu^{(n)}(s,r(\cdot))ds+ \int_{t_0}^t \sigma_r
  \sqrt{r(s)} \, dW_r^{(n)}(s), \quad t\in [t_0,T],
  $$
with
  \begin{align*}
  \mu^{(n)}(s,r(\cdot))&= \mu^\mathbb{P}(s,r(\cdot)) - \sigma_r\,\sqrt{r(s)} \, \xi_r^{(n)}(s) = \mu^\mathbb{P}(s,r(\cdot)) - (\mu^\mathbb{P}(s,r(\cdot))-\mu^\mathbb{Q}(s,r(\cdot)))\, \mathbf{1}_{\{s\leq \tau_n\}}
  \\&= \mu^\mathbb{P}(s,r(\cdot)) \, \mathbf{1}_{\{s> \tau_n\}}
  +\mu^\mathbb{Q}(s,r(\cdot)) \, \mathbf{1}_{\{s\leq \tau_n\}}.
\end{align*}
and therefore
\begin{align*}
 r({t\wedge \tau_n})&=r_0(t_0)+ \int_{t_0}^{t\wedge \tau_n} \mu^{(n)}(s,r(\cdot))ds+ \int_{t_0}^{t\wedge \tau_n} \sigma_r
  \sqrt{r(s)} \, dW_r^{(n)}(s), \quad t\in [t_0,T]\\
&=r_0(t_0)+ \int_{t_0}^{t\wedge \tau_n}\mu^\mathbb{Q}(s,r(\cdot)) ds+ \int_{t_0}^{t\wedge \tau_n} \sigma_r
  \sqrt{r(s)} \, dW_r^{(n)}(s), \quad t\in [t_0,T].
\end{align*}
Consequently, by the already observed  weak uniqueness for the SDDE~\eqref{eq:r-tilde}, the joint probability laws of $\big(\{r(t\wedge\tau_n)\}_{t\in [t_0,T]} ,\tau_n\big)$ under $\mathbb{Q}^{(n)}$ and $\big(\{\widetilde{r}(t\wedge\widetilde{\tau}_n)\}_{t\in [t_0,T]},\widetilde{\tau}_n)$ under $\mathbb{P}$ are equal. Then by~\eqref{eq:Q-n-tau=T},
$$
\mathbb{E}^{\mathbb{P}}[Z_T]
=\lim_{n \rightarrow \infty} \mathbb{Q}^{(n)} (\tau_n=T)= \lim_{n \rightarrow \infty}
\mathbb{P} (\widetilde{\tau}_n=T)=1.
$$
Hence $\mathbb{Q}$ is equivalent to $\mathbb{P}$.
\end{proof}

\begin{lem}\label{lem:LemmaRiccati}${}$\\
Consider the following Riccati equation
\begin{equation}\label{eq:RiccatiEq}\begin{cases}
&\varphi^\prime(t)=\frac{1}{2}\sigma^2
\varphi^2(t)+a\varphi(t)-(1+bg(t)),\quad t\in (-\infty, T)\\
&\varphi(T)=\psi,
\end{cases}
\end{equation}
where the parameter $a$, $b$ and $\sigma$ are positive constants, $g(t)$ is a continuous and nonnegative function, and $\psi \geq 0$.\\
Then, in $(-\infty, T]$,  Eq.~\eqref{eq:RiccatiEq} has a unique solution
$$
\varphi(t)=\Phi(t,T;w,g(\cdot)),
$$
which is positive in $(-\infty, T)$.
Moreover if the function $g(t)=\gamma$ is a nonnegative constant, then the solution is
$ \varphi(t)=\Phi(t,T;\psi, \gamma)$, where
\begin{align}\label{eq:sol-riccati-gamma}
\Phi(t,T;\psi, \gamma)&:=\frac{
\frac{k(\gamma)-a}{\sigma^2}\left(\psi+\frac{a+k(\gamma)}{\sigma^2}\right)+\frac{a+k(\gamma)}{\sigma^2}
\left(\psi+\frac{a-k(\gamma)}{\sigma^2}\right)
\mathit{e}^{-k(\gamma)(T-t)}}
{\left(\psi+\frac{a+k(\gamma)}{\sigma^2}\right) + \left(\frac{k(\gamma)-a}{\sigma^2}-\psi\right)\mathit{e}^{-k(\gamma)(T-t)}}, \quad t\in (-\infty, T],
\end{align}
with
\begin{align}\label{ineq:IPOTESI-2-BOH}
 k(x)=\sqrt{a^2+2(1+bx)\sigma^2},\qquad x\geq 0.
\end{align}
\end{lem}
\begin{proof}${}$\\
It is well-known that any equation of the Riccati type can always be reduced to the second order linear ODE in $[0,+\infty)$ (see, e.g., Polyanin and Zaitsev~\cite{PolZai}) by a suitable substitution. In our case the substitution is
\begin{equation}\label{eq:SubstitutionRiccati}
z(s)=\mathit{e}^{\frac{\sigma^2}{2}\int_{T-s}^T\varphi(u)\,du},
\end{equation}
and  the   equation is
\begin{equation}\label{eq:RiccatiEq2}\begin{cases}
&z^{\prime\,\prime}(s) = (1+bg(T-s)) \frac{\sigma^2}{2} z(s) - a z^\prime(s),\quad s\in (0,\infty)\\
&z(0)=1, \quad z^\prime(0)=  \frac{\sigma^2}{2}  \psi.
\end{cases}\end{equation}
Since Eq~\eqref{eq:RiccatiEq2} has continuous coefficients (and hence bounded on every bounded interval), existence and uniqueness follow by standard results.

Formally, by~\eqref{eq:SubstitutionRiccati},
\[\varphi(t)=\frac{2}{\sigma^2}\,  \frac{z^\prime(T-t)}{z(T-t)},   \quad t\in (-\infty, T), \]
the final condition $\varphi(T)=\psi$ being obviously satisfied.
The above solution is well defined and positive in $(-\infty, T)$ under the (sufficient) condition that $z(s)$ and $z^\prime(s)$ are positive in $(0,\infty)$.

To prove that the unique solution $z(s)$ is positive in $(0,\infty)$, together with its derivative $z^\prime(s)$, we  compare Eq.~\eqref{eq:RiccatiEq2} with the following differential equation
\begin{equation}\label{eq:RiccatiEq0}\begin{cases}
&z^{\prime\,\prime}_0(s) = \frac{\sigma^2}{2} z_0(s) - a z^\prime_0(s),\\
&z_0(0)=1, \quad z^\prime_0(0)  =\frac{\sigma^2}{2}  w,
\end{cases}\end{equation}
where $0\leq w\leq \psi$, i.e., Eq.~\eqref{eq:RiccatiEq2}, with $g(t)=0$ and a (possibly) different initial condition.
\\
Eq.~\eqref{eq:RiccatiEq0} has a unique solution $z_0(s)$
$$
z_0(s)= \frac{ \sigma^2 w+ a +k}{2 k} \, e^{\frac{k-a}{2} \, s} +
\frac{- \sigma^2 w - a +k }{2 k} \, e^{-\frac{k+a}{2} \, s},
$$
where $k=\sqrt{a^2+2\sigma^2}=k(0)$.
 The solution $z_0(s)$  is positive in $(0,\infty)$, together with its derivative
$$
z^\prime_0(s)= \frac{ \sigma^2 w+ a +k }{2 k } \,\frac{k-a}{2}\,  e^{\frac{k-a}{2} \, s} -
\frac{- \sigma^2 w - a +k }{2 k} \, \frac{k+a}{2}\,e^{-\frac{k+a}{2} \, s},
$$
indeed, for  $ s\in (0,\infty)$,   $z^\prime_0(s)> z^\prime_0(0)=\frac{\sigma^2}{2}  w \geq 0$.
Moreover, setting
$$
f(s,x,p)= (1+bg(T-s)) \frac{\sigma^2}{2} x - a p,
$$
and
$$
  P_{f}u(s)= u^{\prime \prime} (s)- f(s, u(s), u^\prime(s)),\quad \text{ for  $u(\cdot)\in C^2(0,\infty)$,}
$$
so that
\[
P_{f}z_0(s)= -bg(T-s) z_0(s),\text{ and } P_{f}z(s)=0,
\]
we can use the Comparison Theorem  3.XVI  in Walter~\cite{Walter} (see p.~139, and in particular inequalities (a') and (b) therein), and assert that
\[z(s)\geq z_0(s)>0, \,\quad z^\prime(s)\geq z_0^\prime(s)> 0, \quad s\in (0,\infty) .\]
Then by~\eqref{eq:SubstitutionRiccati},
\[\varphi(t)=\frac{2}{\sigma^2}\,  \frac{z^\prime(T-t)}{z(T-t)} \geq 0, \; t\in (-\infty, T], \quad \varphi(t)>0, \; t\in (-\infty, T),  \]
i.e., the Riccati equation~\eqref{eq:RiccatiEq} with final condition $\varphi(T)=\psi$ has a unique and positive solution in $t\in (-\infty, T)$.

Finally, in the case $g(t)=\gamma>0$,    Eq.~\eqref{eq:RiccatiEq2} is obtained by replacing $\sigma^2$ with $(1+b\gamma)\,\sigma^2$, and $w$ with $\psi$ in Eq.~\eqref{eq:RiccatiEq0}, so that $z(s)$ and $z^\prime(s)$ are obtained by replacing $k$ with $k(\gamma)$ in  the explicit expressions $z_0(s)$ and $z_0^\prime(s)$.

\end{proof}

\begin{oss}\label{oss:Riccati-parameters}${}$\\
The solution $z(t)$ of Eq.~\eqref{eq:RiccatiEq2} depends on the parameter $\psi$, and has partial derivatives w.r.t.~$\psi$. Therefore also the solution $\varphi(t)$ has partial derivatives w.r.t.~$\psi$.  Similarly, if the function $g(\cdot)$ in the previous Lemma~\ref{lem:LemmaRiccati} depends on a parameter $\mu$, i.e., if we consider $g(t,\mu)$, jointly continuous,  with continuous partial derivative w.r.t.~$\mu$, then, by standard results on differentiability with respect to real parameters (see, e.g. Theorem 13.VI, p.~151 in Walter~\cite{Walter}),  the solution depends also on the parameter $\mu$, and has partial derivatives w.r.t.\@ $\mu$ and $\psi$.
\end{oss}
When $g(t)$ is a nonnegative uniformly bounded function,   one can obtain an upper bound for the solution~$\varphi(t)$ of Eq.~\eqref{eq:RiccatiEq}, under suitable hypotheses on the final condition $\psi$, as shown in the following result.

\begin{lem}\label{lem:COMPARISON-BOH}${}$\\
Assume that
\begin{align}\label{ineq:IPOTESI-BOH}
 0\leq \underline{g}\leq g_1(t)\leq g_2(t) \leq \overline{g},
\quad &\text{and}\quad
0\leq\underline{\psi}\leq \psi_1\leq \psi_2 \leq \overline{\psi}.
\intertext{
  Denote  by
\begin{align*}
 \underline{\varphi}(t)&= \Phi(t,T;\underline{\psi}, \underline{g}),\,\qquad \quad
  \overline{\varphi}(t)=\Phi(t,T;\overline{\psi}, \overline{g}),\FF
  \\
  \varphi_1(t)&=\Phi(t,T; \psi_1, g_1(\cdot)),\,\quad \varphi_2(t)=\Phi(t,T; \psi_2, g_2(\cdot)),\,
\end{align*}
 the positive solutions of the Riccati Equation~\eqref{eq:RiccatiEq} with}
g(t)= \underline{g}, \, g_1(t), \,g_2(t),\, \overline{g},\quad &\text{and}\quad  \psi=  \underline{\psi},\, \psi_1,\, \psi_2, \, \overline{\psi},\notag
\end{align}
  respectively.
\\
 Then, for $t\leq T$,
\begin{equation}\label{ineq:TESI-BOH}
 \underline{\psi} \wedge \frac{k(\underline{g}) -a}{\sigma^2}  \leq\underline{\varphi}(t) \leq
\varphi_1(t) \leq\varphi_2(t) \leq\overline{\varphi}(t)\leq \overline{\psi} \vee \frac{k(\overline{g})  -a}{\sigma^2}.
\end{equation}

\end{lem}

\begin{proof}${}$\\
First of all we observe that when  $g(t)=\gamma$  the solutions $\varphi(t)$ of the Riccati equation~\eqref{eq:RiccatiEq} are decreasing, constant, or increasing  in  $t\in (-\infty,T]$,   when the final condition  $\psi$ is less than, equal to, or greater than $\frac{k(\gamma) -a}{\sigma^2}$, respectively, as can be easily deduced from the explicit expression  $ \varphi(t)=\varphi(t,T;\psi, \gamma)$ (see \eqref{eq:sol-riccati-gamma}).

 Therefore
\begin{equation}\label{eq:inequalities-varphi}
\psi \wedge \frac{k(\gamma)-a}{\sigma^2} \leq \Phi(t,T;\psi, \gamma) \leq \psi \vee \frac{k(\gamma)-a}{\sigma^2},\quad t\in (-\infty,T],
\end{equation}
and the first inequality  in~\eqref{ineq:TESI-BOH}  is immediately achieved, together with the last one.

The other inequalities in~\eqref{ineq:TESI-BOH} can be achieved by standard comparison theorems (see, e.g., Theorem 9.IX and its Corollary in Walter~\cite{Walter} (see p.~96): focusing on the solutions $\varphi_1(t)$ and $\varphi_2(t)$, it is sufficient to note that $\varphi_1(T)\leq \varphi_2(T)$ by~\eqref{ineq:IPOTESI-BOH}, and that $P_2\varphi_1(t)=b[g_2(t)-g_1(t)] \geq 0= P_2\varphi_2(t)$  in $(-\infty, T]$,
where $P_2\varphi(t)=\varphi^\prime(t)- F(\varphi(t)) + b g_2(t)$,
where $F$ is the locally Lipschitz function $F(x)=\frac{1}{2}\,\sigma^2\, x^2+ a x -1$.
\end{proof}

\begin{lem}\label{lem:LemmaTecnico} ${}$\\
Let $a$, $b$ and $\sigma$ be  constants with $a,\,\sigma>0$, and $b\geq0$.
Let  $w \in \left[0,\frac{k(0)-a}{\sigma^2}\right)$,
 where the function~$k(x)$ is defined as in~\eqref{ineq:IPOTESI-2-BOH} of Lemma~\ref{lem:LemmaRiccati}.
Then  the sequence of ordinary differential equations so defined
\begin{equation}\label{eq:SequenceODE0}
\begin{cases}
\frac{d}{d t}\varphi_0(t)=&\frac{1}{2}\sigma^2 \varphi^2_0(t)+
a\,\varphi_0(t)-1
\\ &\phantom{-b\,\alpha(t+\tau,T;w)\quad}\phantom{x}\text{for $T-\tau\leq t\leq T$,}\\
\varphi_0(T)=&w,
\end{cases}
\end{equation}
and
\begin{equation}\label{eq:SequenceODEj}\begin{cases}
\frac{d}{dt}\varphi_j(t)=&\frac{1}{2}\sigma^2\varphi^2_j(t)+a\,
\varphi_j(t)-1-b\,\varphi_{j-1}(t+\tau)\\
&\phantom{-b\,\alpha(t+\tau,T;w)\quad}\phantom{xxxxxxxxxx}\text{for $T-(j+1)\tau\leq t\leq T-j\tau$},\\
\varphi_j(T-j\tau)=&\varphi_{j-1}((T-j\tau)^+).\\
\end{cases}
\end{equation}
has a unique solution $\{\varphi_j(t,T;w), \; j\geq 0\}$.
\\
Morever, setting
\begin{align*}
\overline{\gamma}:= \sup_{t\in[t_0,T]}\gamma(t), \quad
\overline{w}_0&:= \frac{k(0)-a}{\sigma^2}\vee \overline{\gamma}, \quad
\overline{w}_{j+1}:= \overline{w}_{j}\vee\frac{k(\overline{w}_{j})-a}{\sigma^2}, \, j\geq 0,
\\
\underline{\varphi}(t)&= \Phi(t,T;w,0),\quad \qquad \quad  \text{ for $t\in [t_0,T]$,}
 \intertext{and}
   \overline{\varphi}_j(t)&=\Phi(t,T-j\tau;\overline{w}_{j},\overline{w}_{j}),\quad \text{ for $t\in [T-(j+1)\tau,T-j\tau]$,}
\end{align*}
 the solutions satisfy the following inequalities
$$
w \leq \underline{\varphi}(t)\leq \varphi_j(t,T;w)\leq \overline{\varphi}_j(t)\leq  \overline{w}_{j+1}
\quad j\geq 0,\quad \text{for $T-(j+1)\tau\leq t\leq T-j\tau$,}
$$
Furthermore the solutions
 are right-differentiable w.r.t.\@ $w$.
\end{lem}
\begin{proof}${}$\\
The result is achieved by induction with the following steps:
\begin{enumerate}
\item
Since $\varphi_0(t,T;w)=\underline{\varphi}(t)$, for $T-\tau \leq t\leq T$, all the statements hold for $j=0$, by
the following chain of inequalities
    $$w=w \wedge \frac{k(0) -a}{\sigma^2} \leq\frac{k(0) -a}{\sigma^2}\leq \overline{w}_{0}  \leq\overline{w}_{0} \vee \frac{k(\overline{w}_{0}) -a}{\sigma^2}= \overline{w}_{1}, $$
 the previous Lemma~\ref{lem:LemmaRiccati}, Lemma~\ref{lem:COMPARISON-BOH}, with $\underline{g}=g_1(t),= g_2(t)=0$, $\overline{g}=\overline{w}_0$, $\underline{\psi}=\psi_1=\psi_2=w$, $\overline{\psi}=\overline{w}_0$, and  Remark~\ref{oss:Riccati-parameters}.
\item
For $j\geq 1$, in the interval $[T-(j+1)\tau, T-j\tau)$, the $j-th$ differential equation is
\begin{equation}\label{eq:Riccati-F}
\frac{d}{dt}\varphi(t)= F(\varphi(t)) - b g(t,w),
\end{equation}
  with $F(x)=\frac{1}{2}\,\sigma^2\, x^2+ a x -1$,  $g(t,w)=\varphi_{j-1}(t+\tau,T;w)$, and
  final condition
  $$\varphi_j(T-j\tau,T;w)= \varphi_{j-1}((T-j\tau)^+,T;w). $$
   By Lemma~\ref{lem:LemmaRiccati}, we can write
  $$
\varphi_{j}(t,T;w)=\Phi\big(t, T-j \tau;  \varphi_{j-1}((T-j\tau)^+,T;w), \varphi_{j-1}(\cdot+\tau,T;w)\big), \quad j\geq 1.
$$
   In the same interval   $\underline{\varphi}(t)$ satisfies the same equation \eqref{eq:Riccati-F} with  $g(t,w)=0$ and final condition  $\underline{\varphi}(T-j\tau)$;  similarly $\overline{\varphi}_j(t)$   satisfies the same equation with  $g(t,w)=\overline{w}_j$, and  final condition $\overline{w}_j$.

\item Assuming that the statements hold for $j-1$,     the final condition of the $j-th$ system
 is such that
 $$
w\leq\underline{\varphi}(T-j\tau) \leq \varphi_j(T-j\tau,T;w)= \varphi_{j-1}((T-j\tau)^+,T;w)\leq \overline{\varphi}_{j-1}(T-j\tau)\leq  \overline{w}_j,
 $$
 and furthermore
 $$
0(\leq w) \leq \varphi_{j-1}(t+\tau,T;w)\leq \overline{w}_j.
$$
\end{enumerate}
Then
\\\emph{\textbf{(i)}} the inequalities hold for $j$, by applying the previous Lemma~\ref{lem:LemmaRiccati}, Lemma~\ref{lem:COMPARISON-BOH} with $T$ replaced by $T-j\tau$,  and
\begin{align*}
\underline{g}=0, \;\quad \overline{g}=\overline{w}_j,  \qquad \qquad \qquad\qquad &\quad \underline{\psi}=\underline{\varphi}(T-j\tau),  \;\,\quad  \overline{\psi}=\overline{w}_j,\;\quad
 \\
 g_1(t)=g_2(t)=\varphi_{j-1}(t+\tau,T;w),\quad& \quad \psi_1=\psi_2=\varphi_{j-1}((T-j\tau)^+,T;w),
\end{align*}
 so that $\varphi_1(t)=\varphi_2(t)=\varphi_j(t,T;w)$;
 \\
 \emph{\textbf{(ii)}}
  the differentiability properties follow  by Remark~\ref{oss:Riccati-parameters}, taking into account the induction step.

\end{proof}

\begin{oss}\label{rem:LemmaTecnico}
The differential system~\eqref{eq:term_structure_diff_system_delay-BP} has locally Lipschitz coefficients, nevertheless,
when $w\in[0, \frac{k(0)-a}{\sigma^2})$, it is equivalent to a differential system with globally Lipschitz coefficients, thanks to the previous Lemma~\ref{lem:LemmaTecnico}:
\\
 The nondecreasing sequence $\{\overline{w}_j, \, j\geq 0\}$ in
 the previous Lemma~\ref{lem:LemmaTecnico} has a finite limit $L$,
since
$$
x\vee \frac{k(x)-a}{\sigma^2}
=
\begin{cases}
\frac{k(x)-a}{\sigma^2} & x \in [0, \frac{b-a+\sqrt{(b-a)^2 + 2\sigma^2}}{\sigma^2}]\FF
\\
x & \text{otherwise}.
\end{cases}
$$
 and therefore the functions $\alpha^\mathbb{Q}_r(t,T;w)\in [0,L] $, for all $w\in[0, \frac{k(0)-a}{\sigma^2})$, and $t\leq T$.
\end{oss}

   
\end{document}